\documentclass[a4paper, 11pt]{amsart}

\usepackage{amsfonts,mathtools,amssymb,amsthm} 

\usepackage[a4paper,margin=2.5cm]{geometry}
\usepackage{adjustbox}
\usepackage[symbol]{footmisc}
\usepackage{color,enumitem}
\usepackage{microtype}

\usepackage[backend=bibtex,
style=numeric,
autocite=inline,maxbibnames=100,sortcites,doi=false,isbn=false, url=false]{biblatex}
\renewbibmacro{in:}{}

\usepackage{hyperref}

\usepackage{tikz-cd}
\usetikzlibrary{fit,calc}

\numberwithin{equation}{section}

\newtheorem{theorem}{Theorem}[section]
\newtheorem{proposition}[theorem]{Proposition}

\newtheorem*{conjecture*}{Conjecture}
\newtheorem*{theorem*}{Theorem}
\newtheorem{corollary}[theorem]{Corollary}
\newtheorem{lemma}[theorem]{Lemma}

\newtheorem{namedthm}{Theorem}
\newenvironment{thm*}[1]{%
  \renewcommand\thenamedthm{#1}%
  \namedthm}{\endnamedthm}

\newtheorem{namedconj}{Conjecture}
\newenvironment{conj*}[1]{%
  \renewcommand\thenamedconj{#1}%
  \namedconj}{\endnamedconj}

\theoremstyle{definition}
\newtheorem{definition}[theorem]{Definition}
\newtheorem{example}[theorem]{Example}
\newtheorem{remark}[theorem]{Remark}
\newtheorem{assumption}[theorem]{Assumption}
\newtheorem{hypothesis}[theorem]{Hypothesis}

\newcommand{\C}{\mathbb{C}}
\newcommand{\N}{\mathbb{N}}

\newcommand{\R}{\mathbb{R}}
\newcommand{\Z}{\mathbb{Z}}

\DeclareMathOperator{\im}{im}

\DeclareMathOperator{\SL}{SL}

\DeclareMathOperator{\Hom}{Hom}

\DeclareMathOperator{\rank}{rank}
\DeclareMathOperator{\Ho}{H}

\DeclareMathOperator{\vcd}{vcd}
\DeclareMathOperator{\Ext}{Ext}

\DeclareMathOperator{\cd}{cd}
\DeclareMathOperator{\Ann}{Ann}
\DeclareMathOperator{\pd}{pd}
\DeclareMathOperator{\id}{id}
\DeclareMathOperator{\glob}{gldim}
\DeclareMathOperator{\Con}{Con}

\DeclareMathOperator*{\dirlim}{\varinjlim}
\DeclarePairedDelimiter{\ceil}{\lceil}{\rceil}
\DeclarePairedDelimiter{\floor}{\lfloor}{\rfloor}

\newcommand{\G}{\Gamma}
\newcommand{\ZG}{\Z G}
\newcommand{\ZF}{\Z F}
\newcommand{\ZH}{\Z H}
\newcommand{\Br}{\mathrm{Br}}
\newcommand{\ZB}{\Z\Br}
\newcommand{\SLZ}{\SL_{2}(\Z)}

\newcommand{\Gp}{\G^{\prime}}
\newcommand{\Hil}{\mathcal{H}}

\newcommand{\lMod}[1]{#1{\operatorname{-Mod}}}

\newcommand{\catname}[1]{{\normalfont\textbf{#1}}}

\newcommand{\FPinf}{\catname{FP}_{\infty}}
\newcommand{\FP}{\catname{FP}}

\newcommand{\FL}{\catname{FL}}
\newcommand{\om}{\omega}

\newcommand{\RMod}{\lMod{R}}

\newcommand{\SMod}{\lMod{S}}
\newcommand{\ZGmod}{\lMod{\ZG}}
\newcommand{\ZHmod}{\lMod{\ZH}}
\newcommand{\ZGamod}{\lMod{\Z\G}}
\newcommand{\Zmod}{\lMod{\Z}}
\newcommand{\pr}{^{\prime}}
\newcommand{\ma}{\mathfrak{m}}
\newcommand{\loc}{(R,\ma,\Bbbk)}
\newcommand{\Pol}{\overline{P}_{\bullet}}
\newcommand{\Fox}[2]{\frac{\partial #1}{\partial #2}}
\newcommand{\Foxi}[1]{\frac{\partial #1}{\partial y_{i}}}
\newcommand{\foxi}{\Foxi{ }}
\newcommand{\od}{\overline{d}}
\newcommand{\Ngk}[2]{\mathbf{N}_{#1}^{#2}}

\newcommand{\so}{\sigma_{1}}
\newcommand{\st}{\sigma_{2}}
\newcommand{\sh}{\sigma_{3}}

\newcommand{\opp}{\overline{P}}
\newcommand{\off}{\overline{F}}
\newcommand{\Fol}{\overline{F}_{\bullet}}
\newcommand{\ZGk}{\Z\G_{k}}
\newcommand{\oph}{\overline{\varphi}}
\newcommand{\bul}[1]{{#1}_{\bullet}}
\newcommand{\ohh}[1]{\overline{h}_{#1}}
\newcommand{\Pquo}{\bul{P}/x\bul{P}}

\title{Eventually periodic resolutions with applications to integral group rings}
\author{Sean P. Carroll}
\address{Department of Mathematics, Northeastern University, 43 Leon Street, Boston, MA 02115, USA}
\email{carroll.sea@northeastern.edu}

\begin{document}
    \makeatletter
    \@namedef{subjclassname@2020}{\textup{2020} Mathematics Subject Classification}
    \makeatother

    \date{\today}

    \begin{abstract}
        We present a general construction of eventually periodic projective resolutions for modules over the quotient of a ring by a regular central element. Our approach combines the iterated mapping cone technique with the higher homotopies of Shamash. The periodic part of each resolution comes from a matrix factorization of the central element. The construction is applied to the integral group rings of quotients of groups of finite cohomological dimension by infinite cyclic central subgroups. We give explicit calculations for several families of groups, including hyperbolic triangle groups and the mapping class group of the plane with four punctures.
    \end{abstract}
    
    \maketitle

    \section{Introduction}
    The study of projective resolutions and their periodic behavior has long been a central theme in both commutative algebra and the cohomology of groups. A projective resolution is \emph{eventually $q$-periodic} if its modules and differentials repeat with period $q$ from some degree $k\ge0$ onward. A module admits such a resolution if and only if it has a projective resolution whose syzygy modules satisfy $\Omega^{k}\cong\Omega^{k+q}$ for some $k$, and in this case the resolution can be spliced into an eventually periodic one, as described below. Our aim in this paper is to construct such resolutions explicitly over a quotient ring $S=R/(x)$ by a regular central element, and then to compute with them over integral group rings.

    Our motivation comes from both of these directions. In commutative algebra one of the foundational results is due to Eisenbud \cite[Thm.~6.1, p.52]{matfac}.
    \begin{thm*}{A}[Eisenbud~1980]\label{thm:ebud}
        Let $\loc$ be a commutative regular local ring, let $x\in\ma$ and let $S=R/(x)$ be the quotient ring. Then the minimal free resolution of every finitely generated $S$-module becomes $2$-periodic after $\dim(R)+1$ steps.
    \end{thm*}
    The existence of these eventually $2$-periodic resolutions gave rise to the theory of \emph{matrix factorizations}, and to an equivalence between the stable category of maximal Cohen--Macaulay modules over $S$ and the homotopy category of matrix factorizations of $x$. An account of this equivalence is given in the book of Yoshino \cite{yoshino}, and a survey is given by Martsinkovsky \cite{mart}.

    In the cohomology of groups, periodicity of a projective resolution records periodicity of the cohomology functors. Suppose that $\bul{P}\to\Z$ is a projective resolution over $\ZG$ whose syzygy modules satisfy $\Omega^{k}\cong\Omega^{k+q}$. Then the exact sequence
    \[
        \begin{tikzcd}[column sep=1.6em]
            0 \arrow[r] & \Omega^{k+q} \arrow[r] & P_{k+q-1} \arrow[r] & \cdots \arrow[r] & P_{k} \arrow[r] & \Omega^{k} \arrow[r] & 0
        \end{tikzcd}
    \]
    can be spliced onto the tail of $\bul{P}$ infinitely many times, and the result is a projective resolution of $\Z$ which is periodic of period $q$ after $k$ steps. It follows that there are natural equivalences $\Ho^{i}(G,-)\cong\Ho^{i+q}(G,-)$ for all $i>k$, and under Yoneda's interpretation of $\Ext$ these equivalences are given by product with the invertible class in $\Ext^{q}_{\ZG}(\Omega^{k},\Omega^{k+q})$ represented by the displayed sequence \cite{yoneda}. Talelli showed that, conversely, natural equivalences $\Ho^{i}(G,-)\cong\Ho^{i+q}(G,-)$ for all $i>k$ imply the existence of a projective resolution which is periodic after $k$ steps \cite[Thm.~1.1]{talelli1}. The subject itself goes back to Artin--Tate (unpublished, but see \cite[\S XII]{cartaneilenberg}), Swan \cite{swanperiod} and Milnor \cite{sphere}, who related periodicity in the Tate cohomology of a finite group to free actions on homotopy spheres.

    Now, every group $G$ acts freely on some CW-complex which is homotopy equivalent to a sphere. One way to build such a space is to begin with the universal cover $X=\widetilde{K(G,1)}$ of a classifying space for $G$. This space is contractible, and for any $n>0$ the product $\mathcal{S}_{G}=X\times S^{n}$ is a free $G$-CW-complex homotopy equivalent to $S^{n}$, where $G$ acts by deck transformations on the first factor and trivially on the second. Since no finiteness condition was placed on $X$, the space $\mathcal{S}_{G}$ is in general infinite-dimensional. Wall \cite{wallppr}, Mislin--Talelli \cite{mislintalelli} and Talelli \cite{talelliperiodiccohomology, talelli2005} survey the problems surrounding these spaces, and formalize the above observations into the following.
    \begin{conj*}{A}\label{Conj}
        A group $G$ has eventually periodic cohomology if and only if $G$ admits a finite-dimensional free $G$-CW-complex that is homotopy equivalent to a sphere.
    \end{conj*}
    That is, eventual periodicity of the cohomology of $G$ is conjectured to be the algebraic characterization of the existence of a \emph{finite-dimensional} $\mathcal{S}_{G}$. The passage from such an action to eventual periodicity is straightforward, and goes back to Wall \cite{wallppr} (see also \cite[Cor.~5.2]{mislintalelli}), while the converse remains open in general. Johnson first proved a special case of the converse for groups fitting into a short exact sequence $1\to K\to G\to H\to1$ with $K$ of type $\FP$ and $H$ a finite group with periodic cohomology \cite{johnson}. Connolly and Prassidis then settled the case when $G$ has finite virtual cohomological dimension \cite{connolly}. Mislin and Talelli proved the conjecture for groups in Kropholler's class $\mathrm{H}\mathfrak{F}$ \cite{KrophollerFP} with a bound on the orders of their finite subgroups \cite[Thms.~A and B]{mislintalelli}. This class is closed under extensions and contains every group of finite virtual cohomological dimension, and Talelli later removed the bound \cite[Cor.~3.5]{talelli2005}. Adem and Smith proved the conjecture under the assumption that the periodicity is induced by cup product with an invertible cohomology class of positive degree \cite{ademsmith}. For more about this viewpoint we suggest the account of Jo \cite{joalginv}.

    Each group treated in Section~\ref{S:examples} has finite virtual cohomological dimension, and its cohomology is eventually periodic by Remark~\ref{Rem:gysin}. Then by \cite[Thm.~1.1]{talelli1} these groups are already known to have eventually periodic resolutions, and our concern is instead with the structure of such a resolution, with how small it can be taken, and with the constraints satisfied by its ranks.

    \subsection{Main results}
    The present work gives a general construction of eventually $2$-periodic projective resolutions over a quotient ring $S=R/(x)$, where $x$ is a regular central non-unit of $R$ and $R$ itself need not be commutative. Beyond the existence of $x$, no hypothesis is placed on the ring.

    Let $0\neq N\in\FPinf(S)$ be a module of type $\FPinf$ (Definition~\ref{Def:FPn}) whose projective dimension $n=\pd_{R}(N)$ over $R$ is finite, and let $\bul{P}=(P_{i},d_{i})_{0\le i\le n}$ be a finite type $R$-projective resolution of $N$ of length $n$. By reducing modulo $x$ we obtain a complex $\Pol\coloneqq\Pquo$ of projective $S$-modules, with differentials $\od_{i}$.

    Since $x$ annihilates $N$, multiplication by $x$ is null-homotopic on $\bul{P}$, and we fix a null-homotopy $\{\varphi_{i}:P_{i-1}\to P_{i}\}_{i\ge1}$ for it (null-homotopies of this kind go back to Shamash \cite[Lem.~1, p.454]{shamash1}). Then by reducing this null-homotopy modulo $x$ we obtain a chain map
    \[\Phi=\{\oph_{i}\}_{i\ge1}:\Sigma\Pol\longrightarrow\Pol,\]
    and in Lemma~\ref{Lem:shamash} we show that $\Phi$ induces an isomorphism on first homology. We write $\Phi^{2}=0$ for the condition that $\oph_{i}\oph_{i+1}=0$ for all $i\ge 1$ (see Hypothesis~\ref{Hyp:star}).
    \begin{thm*}{1}[\textbf{= Theorems~\ref{Thm:imc} \& \ref{Thm:main}}]\label{thm:1}
        Suppose that $\Phi^{2}=0$. Then $N$ has a finite type $S$-projective resolution $\bul{M}$, which is $2$-periodic in degrees $\ge n-1$, with
        \[M_{j}=\bigoplus_{\substack{0\le a\le\min(j,n)\\ a\equiv j\!\!\pmod{2}}}\opp_{a}\qquad(j\ge0)\]
        and with at most two non-zero components of the differential out of each summand, namely
        \[\od_{a}:\opp_{a}\longrightarrow\opp_{a-1},\qquad -\oph_{a+1}:\opp_{a}\longrightarrow\opp_{a+1}.
        \]
    \end{thm*}
    Here a component $\opp_{a}\to\opp_{b}$ is to be read as zero whenever $\opp_{b}$ is not a summand of $M_{j-1}$. For a given choice of $(\bul{P},\{\varphi_{i}\}_{i\ge1})$, every module and every differential of $\bul{M}$ is built from the modules $P_{0},\ldots,P_{n}$ and the finite sets of maps $\{d_{1},\ldots,d_{n}\}$ and $\{\varphi_{1},\ldots,\varphi_{n}\}$. Following Nguyen and Veliche \cite[\S3]{imc}, the resolution arises as the direct limit of a sequence of iterated mapping cones taken along suspensions of the chain map $\Phi$. The recursion involved can be solved explicitly, and the whole resolution is determined by the first $\ceil*{\frac{n}{2}}$ mapping cones.

    The hypothesis $\Phi^{2}=0$ is a condition on the chosen pair $(\bul{P},\{\varphi_{i}\}_{i\ge1})$, and not on the module $N$. When $n\le2$ it holds for any null-homotopy, and when $n=3$ it reduces to the single equation $\oph_{1}\oph_{2}=0$, by Proposition~\ref{Prop:starlow}. In general the hypothesis can be described in terms of matrix factorizations of $x$ over $R$, and this is the subject of Section~\ref{S:whenstar}. When the hypothesis fails, a resolution with the same modules as in Theorem~\ref{thm:1} can still be built from a system of higher homotopies in the sense of Shamash, as we show in Section~\ref{S:higher}.

    We then specialize to integral group rings. Let $G$ be a group of finite cohomological dimension (Definition~\ref{Def:FP}) with a central element $1\neq c\in Z(G)$ and set $\G=G/\langle c\rangle$. Then $x=1-c$ is a regular central non-unit of $\ZG$ with $\Z\G\cong\ZG/(x)$, and we are in the setting of Theorem~\ref{thm:1}. Now suppose that $G$ is of type $\FL$. Since $c$ is a non-trivial central element, it follows from the algebraic form of Gottlieb's theorem due to Stallings \cite{stallingsgott} that $\chi(G)=0$. Then the even and odd rank sums of a finite type free resolution of $\Z$ over $\ZG$ agree, and hence the free rank in the periodic part is constant.
    \begin{thm*}{2}[\textbf{= Theorem~\ref{Thm:zg}}]\label{thm:2}
        With the notation as above:
        \begin{enumerate}[label=\textup{(\arabic*)}]
            \item Every $\Z\G$-module of type $\FPinf$ has a finite type $\Z\G$-projective resolution that is eventually $2$-periodic.
            \item If $G$ is of type $\FL$ with $n=\cd(G)<\infty$ and $\bul{F}\to\Z$ is a finite type free $\ZG$-resolution of length $n$, then the trivial module $\Z\in\ZGamod$ has a finite type free resolution which is $2$-periodic in degrees $\ge n-1$ and of constant rank
            \[r=\sum_{i\equiv0\!\!\pmod{2}}\rank_{\ZG}(F_{i})=\sum_{i\equiv1\!\!\pmod{2}}\rank_{\ZG}(F_{i})\]
            in those degrees.
        \end{enumerate}
    \end{thm*}
    In a chosen basis the differentials are matrices over $\Z\G$, and by applying the augmentation map $\varepsilon:\Z\G\to\Z$ to their entries we obtain integer matrices whose Smith normal forms determine $\Ho_{\ast}(\G,\Z)$ and $\Ho^{\ast}(\G,\Z)$.

    In the situation of Theorem~\ref{thm:2}(2), the higher homotopies of Section~\ref{S:higher} assemble into a matrix factorization of $x=1-c$ over $\ZG$ of rank $r$, and the $2$-periodic part of the resolution is isomorphic to its reduction modulo $x$ (Corollary~\ref{Cor:mfzg}). When the null-homotopy satisfies the stronger condition $\varphi_{i}\varphi_{i+1}=0$ for all $i\ge 1$, this matrix factorization is assembled from the differentials of $\bul{F}$ and the maps $\{\varphi_{i}\}_{i\ge1}$ alone, and this is the case in all of the calculations below.

    \vskip 5pt
    \noindent \textbf{Examples (= Section~\ref{S:examples}).}
    \begin{enumerate}
        \item Amalgamated products of cyclic groups (= Theorem~\ref{Thm:tknot}): for $p,q,k\in\N$ with $k\ge1$, $p,q\ge2$ and $\gcd(p,q)=1$, the group $\G_{k}\cong \Z/(kp) *_{\Z/k} \Z/(kq)$ has a $2$-periodic free resolution in degrees $\ge 1$, with stable rank $2$. The same argument covers an amalgam of finitely many finite cyclic groups over a common subgroup (Remark~\ref{Rem:manycyclics}).

        \item Central quotients of Heisenberg groups (= Theorem~\ref{Thm:hberg}): for $k\in\N$ with $k\ge1$, the group $\Hil_{k}\coloneqq\langle a,b,c\mid [a,b]=c,\; [a,c]=[b,c]=c^{k}=1\rangle$ has a $2$-periodic free resolution in degrees $\ge 2$, with stable rank $4$.
            
        \item Hyperbolic triangle groups (= Theorem~\ref{Thm:tri}): for $p,q,r\in\N$ with $2\le p\le q\le r$ and $\frac{1}{p}+\frac{1}{q}+\frac{1}{r}<1$, the group $T=T(p,q,r)\coloneqq\langle a,b\mid a^{p}=b^{q}=(ab)^{r}=1\rangle$ has a $2$-periodic free resolution in degrees $\ge 2$, with stable rank $3$. The ranks of this resolution are $1$, $2$ and then $3$ in all higher degrees, against $3$, $5$, $4$ and then $3$ for the resolution of Ellis and Williams \cite[Eq.~(4), p.576]{ellistri}.

        \item Central quotients of Artin groups of spherical type (= \S\ref{ss:artin}, Theorem~\ref{Thm:mcgperiod}): for any irreducible Artin group $\mathcal{A}$ of spherical type and rank $s\ge2$, the quotient $\mathcal{A}/Z(\mathcal{A})$ has a $2$-periodic free resolution in degrees $\ge s-1$, with stable rank $2^{s-1}$. In type $A_{m}$ this is the mapping class group of the punctured plane $\mathrm{Mod}(\mathbb{P}_{m+1})\cong \Br_{m+1}/Z(\Br_{m+1})$.
        \end{enumerate}
    \vskip 5pt

    \subsection{Relation to earlier work}\label{ss:earlier}
    Resolutions of the form in Theorem~\ref{thm:1}, built from the modules $\opp_{a}$ of the reduced complex, go back to Shamash \cite{shamash1}, who resolved a module over a hypersurface $S=R/(x)$ by transporting a minimal free resolution over the commutative local ring $R$ across the quotient map. In \cite[\S3]{shamash1} an arbitrary null-homotopy is also enlarged into a system of \emph{higher homotopies}, which assemble into a resolution over $S$ whose differential has a component for each higher homotopy. Eisenbud extended this construction to regular sequences \cite[\S7]{matfac}, and Eisenbud and Peeva have since made it the starting point of their theory of higher matrix factorizations over a complete intersection \cite[\S3.4, Constr.~4.1.3 and \S4.3]{ebudpeeva}. Bergh, Jorgensen and Moore give a converse to the construction of Shamash and Eisenbud over a commutative local ring \cite{bjm}. Under the hypothesis $\Phi^{2}=0$, Theorem~\ref{thm:1} gives a resolution with the same modules as that of Shamash, but whose differential has at most two components out of each summand, and in Section~\ref{S:higher} we record the general construction over a ring which need not be commutative.

    The constructions of Shamash and Eisenbud take place over a commutative local ring, where the maximal ideal provides the minimality of the resolution $\bul{P}$. When $x\notin\ma^{2}$, the maximal ideal also provides the null-homotopy of Remark~\ref{Rem:shamashprop1}. An integral group ring has no distinguished maximal ideal, and its finitely generated modules need not have projective covers (see \cite[Thm.~2.1 and Lem.~2.2]{bass60}). For this reason we work with projective resolutions over an arbitrary ring, where there is in general no notion of a minimal resolution, and even over a local ring the resolution of Theorem~\ref{thm:1} need not be minimal. In particular, the bound $\dim(R)+1$ of Theorem~\ref{thm:ebud}, which concerns minimal resolutions, is not directly comparable with the bound $\pd_{R}(N)-1$ of Theorem~\ref{thm:1}.

    Iterated mapping cones were used in a similar way by Nguyen and Veliche \cite{imc}, who applied them to the Koszul complex on a minimal generating set of the maximal ideal of a complete intersection local ring. Their direct limit is a minimal free resolution of the residue field \cite[Thm.~4.1]{imc}, which is obtained independently of Tate's construction and has its differentials displayed as blocks of matrices. The recursion of Section~\ref{S:gencons} follows theirs, with $\Pol$ in place of the Koszul complex and the suspensions of $\Phi$ in place of their maps built from cycles representing a basis of the first Koszul homology. Chain maps of the same kind as $\Phi$ appear in the unpublished notes of Webb \cite{webb2002}, and iterated mapping cones are used in a related way in \cite[p.72]{nhi2}.

    The theory of matrix factorizations has already been extended beyond commutative rings in several directions. Cassidy, Conner, Kirkman and Moore \cite{ccmk} introduced twisted matrix factorizations of a regular normal element of a connected graded algebra. When the algebra has finite global dimension it is shown in \cite[Thm.~4.2]{ccmk} that minimal resolutions over the quotient become periodic after finitely many steps, provided that the normalizing automorphism has finite order. Mori and Ueyama \cite{moriueyama} extended the notion to an arbitrary non-zero non-unit element of a ring, and Chen \cite{chenmf} developed a general theory of module factorizations for a regular normal element of an arbitrary ring. The results of these works are largely categorical, and identify stable categories of factorizations with stable categories of totally reflexive or Gorenstein projective modules over the quotient ring, or with its singularity category. In contrast, the construction of this paper produces explicit matrix factorizations. For a group $G$ as in Theorem~\ref{thm:2}(2), these are the matrix factorizations of $1-c$ over $\ZG$ of rank $r$ described after Theorem~\ref{thm:2}. That is, we obtain a pair of $r\times r$ matrices over $\ZG$ whose products in either order are multiplication by $1-c$, and these matrices are computed explicitly in \S\S\ref{ss:amalgam}--\ref{ss:triangle} and in Example~\ref{Ex:braid4}.

    Explicit resolutions over integral group rings can also be obtained geometrically. Wall's perturbation technique builds a free resolution for an extension of groups from free resolutions of the normal subgroup and of the quotient \cite[Thm.~1, p.253]{wallresex}. Ellis and Williams adapt this technique to a group acting cellularly on a contractible complex, combining free resolutions of the cell stabilizers into a free resolution of the whole group \cite[\S2]{ellistri}. Ellis has developed this approach into an algorithmic framework \cite{ellis}, and in \S\ref{ss:triangle} we compare the resolution obtained in this way for the hyperbolic triangle groups \cite[\S\S3--4]{ellistri} with that of Theorem~\ref{Thm:tri}.

    \subsection{Organization of the paper}
    The paper is organized as follows. After fixing conventions in Section~\ref{S:notation}, we prove two change-of-rings results in Section~\ref{S:overview} and construct the chain map $\Phi$ in Section~\ref{S:maps}. The iterated mapping cone construction is carried out in Section~\ref{S:gencons} under Hypothesis~\ref{Hyp:star}, and the resulting recursion is solved in closed form in Section~\ref{S:period}. In Section~\ref{S:whenstar} we discuss when the hypothesis holds and relate it to matrix factorizations, and in Section~\ref{S:higher} we use the higher homotopies of Shamash to construct a resolution with the same modules when it fails to hold. The results are applied to integral group rings in Section~\ref{S:applications}, and Section~\ref{S:examples} contains the explicit calculations. We close with some directions for further work.

    \section{Notation and conventions}\label{S:notation}
    Let $R$ be an associative ring with identity. The category of \emph{left} $R$-modules is denoted by $\lMod{R}$. By an \emph{$R$-module} we will always mean a \emph{left} $R$-module.

    In order to keep the convention of matrices acting on the opposite side of scalars, maps of free modules $R^{p}\to R^{q}$ are represented by $p\times q$ matrices with entries in $R$ that act on the \emph{right} of \emph{row vectors}. That is to say, if $f:R^{p}\to R^{q}$ is represented by the $p\times q$ matrix $A$ and $v=(v_{1},\ldots, v_{p})\in R^{p}$ then
    \[(v)f \coloneqq (v_{1},\ldots, v_{p})A.\]
    Furthermore, maps are composed from \emph{left to right}, so that the composition of $f$ followed by $g$ is written $fg$. In this way the composition of maps of free $R$-modules is calculated by ordinary matrix multiplication.

    We write $\N=\{0,1,2,\ldots\}$ for the natural numbers.

    A \emph{complex} of $R$-modules is a sequence $\bul{C}=(C_{j},\partial_{j})_{j\in\Z}$ of $R$-modules and morphisms $\partial_{j}:C_{j}\to C_{j-1}$ with $\partial_{j+1}\partial_{j}=0$ for all $j$. We say that $\bul{C}$ is \emph{$2$-periodic in degrees $\ge k$} if $C_{j}=C_{j+2}$ for all $j\ge k$ and $\partial_{j}=\partial_{j+2}$ for all $j\ge k+1$, and that it is \emph{eventually $2$-periodic} if this holds for some $k\ge0$. The two ranges differ by one because $\partial_{k}$ maps to $C_{k-1}$, which lies below the periodic range.

    We will use $\Sigma$ to denote the \emph{suspension} of a complex of $R$-modules, which for a complex $\bul{C}$ is the complex $\Sigma\bul{C}$ with modules and morphisms given by
    \[(\Sigma C)_{j} = C_{j-1},\qquad\partial_{j}^{\Sigma C} = -\partial^{C}_{j-1}.\]
    Following \cite[\S3.1]{imc}, the suspension of a chain map $\bul{f}:\bul{A}\to\bul{B}$ is
    \begin{equation}\label{Eq:susmap}
        (\Sigma f)_{j} = -f_{j-1}:(\Sigma A)_{j}\to(\Sigma B)_{j},
    \end{equation}
    which is again a chain map, and by iterating we obtain $(\Sigma^{i}f)_{j}=(-1)^{i}f_{j-i}$. With this sign convention compositions satisfy $\Sigma(fg)=-(\Sigma f)(\Sigma g)$.

    Let $\bul{f}:\bul{A}\to\bul{B}$ be a chain map of complexes of $R$-modules. The \emph{mapping cone} of $\bul{f}$ is the chain complex $\Con(\bul{f})$ with 
    \begin{equation}\label{Eq:mcdiff}
        \Con(\bul{f})_{i}=A_{i-1}\oplus B_{i}, \qquad \partial_{i}=\begin{bmatrix}
        -d^{A}_{i-1} & -f_{i-1} \\
        0 & d^{B}_{i}
        \end{bmatrix}.
    \end{equation}
    This is the mapping cone of \cite[1.5.1, p.18]{weibel}, written in the row-vector convention above. By \cite[1.5.2, p.19]{weibel}, there is an associated short exact sequence of complexes of $R$-modules
    \begin{equation}\label{Eq:sesmc}
		\begin{tikzcd}
			0 \arrow[r] & \bul{B} \arrow[r,"\eta"] & \Con(\bul{f}) \arrow[r,"\nu"] & \Sigma\bul{A} \arrow[r] & 0,
		\end{tikzcd}
	\end{equation}
    where $\eta$ is the inclusion of $\bul{B}$ as a direct summand in each degree and $\nu$ is the projection to $\Sigma\bul{A}$. We write $\iota$ for the graded injection of degree $0$ of $\Sigma\bul{A}$ into $\Con(\bul{f})$, which is not a chain map, and which satisfies $\iota\nu=\id$.

    \begin{assumption}\label{ass:ring}
        We fix throughout this paper the following data:
        \begin{enumerate}
            \item $R$ is an associative ring with identity.
            \item $x\in Z(R)$ is a \emph{regular} central element which is not a unit, that is, $x$ is a central non-zero-divisor with $xR\neq R$.
            \item $S\coloneqq R/(x)$ is the quotient ring. 
        \end{enumerate}
    \end{assumption}
    
     We will always work with $S$-modules that have finite projective dimension over $R$, so nothing is assumed about the global dimension of $R$.

    We will use the following observation repeatedly.
    \begin{remark}\label{Rem:xinj}
        Since $x\in R$ is a non-zero-divisor, multiplication by $x$ is injective on $R$, and hence on any projective $R$-module. Moreover, since $x$ is central, $xP$ is a submodule of any $R$-module $P$, and if $P$ is projective then multiplication by $x$ is an isomorphism $P\to xP$ which preserves any grading on $P$.
    \end{remark}
    
    The main rings of interest are integral group rings $R=\ZG$ with $G$ an infinite discrete group. Modules over these rings are abelian groups with a (left) $G$-action. For $g,h\in G$ we write $[g,h]\coloneqq ghg^{-1}h^{-1}$ for the commutator. Such a ring is rarely noetherian (see \cite[Part III]{passman}), and for this reason we instead place restrictions on the modules we study, to ensure finite generation of successive kernels in a projective resolution. We follow the definitions given by Bieri \cite[\S I.1]{bieri}:
    \begin{definition}\label{Def:FPn}
        For $n\in \N$, we say that a module $N\in\RMod$ is \emph{of type $\FP_{n}$} over $R$ if there is a projective resolution $\bul{P}\to N$ with $P_{i}\in\RMod$ finitely generated for $i\le n$. If $N$ is of type $\FP_{n}$ for all $n\ge 0$, equivalently if $N$ admits a single projective resolution $\bul{P}\to N$ with every $P_{i}$ finitely generated, then $N$ is \emph{of type $\FPinf$} over $R$ and such a resolution $\bul{P}$ is called \emph{finite type}.
    \end{definition}
    We write $\FPinf(R)$ for the full subcategory of modules of type $\FPinf$. If $N\in\FP_{m}(R)$ then $N$ has a \emph{free} resolution which is finitely generated in degrees $\le m$, so a module of type $\FPinf$ has a free resolution of finite type, though not necessarily of finite length \cite[\S VIII.6]{brown}.

    \section{Overview and preliminaries}\label{S:overview}
    In this section we outline the main steps in our construction of projective resolutions by iterated mapping cones, and we prove the two change-of-rings results that the construction will need.

    Let $0\neq N\in\FPinf(S)$ and let $\bul{P}$ be a finite type $R$-projective resolution of $N$ of length $n<\infty$. The construction of an $S$-projective resolution of $N$ proceeds as follows. In Lemma~\ref{Lem:tori} we show that the reduced complex $\Pol$ has homology concentrated in degrees $0$ and $1$, where it is $N$. In Lemma~\ref{Lem:shamash} we produce a chain map $\Phi:\Sigma\Pol\to\Pol$ which induces an isomorphism on first homology. Iterated mapping cones taken along the suspensions of $\Phi$ then `purge' the homology in degree $1$ into successively higher degrees, where in the limit it disappears altogether. The desired $S$-projective resolution is the direct limit of this sequence of mapping cones, constructed in Section~\ref{S:gencons}.

    From Assumption~\ref{ass:ring} there is a short exact sequence of $R$-modules
	\begin{equation}\label{Eq:resquo}
		\begin{tikzcd}
			0 \arrow[r] & R \arrow[r,"\lambda_{x}"] & R \arrow[r] & S \arrow[r] & 0
		\end{tikzcd}
	\end{equation}
	where $\lambda_{x}$ denotes multiplication by $x$ (since $x$ is central, left and right multiplication by $x$ agree). This is an $R$-free resolution of $S\in\RMod$.
	
	The chain map $\Phi$ of Section~\ref{S:maps} is built from the null-homotopy given by the following lemma.
	\begin{lemma}\label{Rem:resquo}
		Let $N\in\RMod$ with $x\in\Ann_{R}(N)$ (equivalently $N\in\SMod$), and let $\bul{P}\to N$ be an $R$-projective resolution. Then multiplication by $x$ lifts to a chain map $\lambda_{x}:\bul{P}\to\bul{P}$, and every such lift is null-homotopic.
	\end{lemma}
	\begin{proof}
		Since $x$ lies in the center of $R$, multiplication by $x$ is an $R$-module endomorphism $\lambda_{x}$ of every $R$-module. Then by the Comparison Theorem \cite[Thm.~2.2.6, p.35]{weibel} the map $\lambda_{x}:N\to N$ lifts to a chain map $\bul{P}\to\bul{P}$, and any two such lifts are homotopic. Since $x$ annihilates $N$, one of these lifts is the zero chain map, and hence every lift is null-homotopic.
	\end{proof}

    For $N\in\SMod$ with an $R$-projective resolution $\bul{P}\to N$ we write
    \[\Pol\coloneqq\Pquo\cong S\otimes_{R}\bul{P}\]
    for the associated complex of projective $S$-modules.
    \begin{lemma}\label{Lem:tori}
		Let $\bul{P}$ be a projective resolution of $N\in\RMod$ and suppose $x\in\Ann_{R}(N)$. Then the homology of $\Pol$ is concentrated in degrees $0$ and $1$, where
		\[\Ho_{0}(\Pol)\cong N\cong\Ho_{1}(\Pol).\]
        If moreover $\bul{P}$ is finite type, then $\Pol$ is a complex of finitely generated projective $S$-modules.
	\end{lemma}
	\begin{proof}
        Since $x\in R$ is a regular central element, multiplication by $x$ is an injective endomorphism of $\bul{P}$ by Remark~\ref{Rem:xinj}. This gives a short exact sequence of complexes of $R$-modules
        \[
			\begin{tikzcd}
				0 \arrow[r] & \bul{P} \arrow[r,"\lambda_{x}"] & \bul{P} \arrow[r] & \Pol \arrow[r] & 0.
			\end{tikzcd}
		\]
        Now $\Ho_{j}(\bul{P})=0$ for $j>0$ and $\Ho_{0}(\bul{P})=N$, while $\lambda_{x}$ is zero on $N$. Then the associated long exact sequence reduces to
        \[
            \begin{tikzcd}[sep=small]
				0 \arrow[r] & \Ho_{j}(\Pol) \arrow[r] & 0
			\end{tikzcd} 
            \quad (j\ge 2),\qquad
            \begin{tikzcd}[sep=small]
				0 \arrow[r] & \Ho_{1}(\Pol) \arrow[r] & N \arrow[r,"0"] & N \arrow[r] & \Ho_{0}(\Pol) \arrow[r] & 0,
			\end{tikzcd}
        \]
        which gives $\Ho_{j}(\Pol)=0$ for $j\ge 2$, together with
        \[\Ho_{1}(\Pol)\cong\ker(\lambda_{x}:N\to N)=N,\qquad \Ho_{0}(\Pol)\cong N/xN=N,\]
        as required.

        Now suppose that $\bul{P}$ is finite type. Since $R^{k}/xR^{k}\cong S^{k}$, the functor $S\otimes_{R}-$ takes finitely generated projective $R$-modules to finitely generated projective $S$-modules, so $\Pol$ is a complex of finitely generated projective $S$-modules.
	\end{proof}
    The isomorphism $\Ho_{1}(\Pol)\cong N$ produced above is a connecting map, and we record its description for later use in Lemma~\ref{Lem:shamash}.
    \begin{remark}\label{Rem:connecting}
        The isomorphism $\delta:\Ho_{1}(\Pol)\xrightarrow{\ \sim\ }\ker(\lambda_{x})=N$ of Lemma~\ref{Lem:tori} is computed as follows. Given a cycle $\overline{z}\in\ker(\od_{1})\subseteq\opp_{1}$, lift it to $z\in P_{1}$. Then $zd_{1}\in xP_{0}$, so that $zd_{1}=x\cdot w$ for some $w\in P_{0}$, and $\delta[\overline{z}]=\varepsilon(w)$, where $\varepsilon:P_{0}\to N$ is the augmentation map of $\bul{P}$.
    \end{remark}
    
    \begin{lemma}\label{Lem:fpinfquo}
        Let $n\in\N$. If $N\in\FP_{n}(S)$, then $N\in\FP_{n}(R)$. In particular, if $N\in\FPinf(S)$, then $N\in\FPinf(R)$.
    \end{lemma}
    \begin{proof}
        We induct on $n\in\N$. For $n=0$, since $S=R/(x)$ is a quotient of $R$, any finitely generated $S$-module is finitely generated over $R$.

        Assume the statement for $n-1$ and let $N\in\FP_{n}(S)$. Then the first step of a free resolution of $N$ which is finitely generated in degrees $\le n$ gives a short exact sequence of $S$-modules
        \[
            \begin{tikzcd}
                0 \arrow[r] & K \arrow[r] & S^{a} \arrow[r] & N \arrow[r] & 0
            \end{tikzcd}
        \]
        with $a<\infty$ and $K\in\FP_{n-1}(S)$. By the inductive hypothesis $K\in\FP_{n-1}(R)$. The finite type free resolution \eqref{Eq:resquo} exhibits $S$ as a module of type $\FPinf$ over $R$, and hence so is $S^{a}$. By \cite[Prop.~1.4, pp.12--13]{bieri}, if $0\to A\to B\to C\to0$ is a short exact sequence of $R$-modules with $A$ of type $\FP_{n-1}$ and $B$ of type $\FP_{n}$, then $C$ is of type $\FP_{n}$. Applied to the sequence above, this gives $N\in\FP_{n}(R)$.

        Finally $\FPinf(R)=\bigcap\limits_{n\ge0}\FP_{n}(R)$, and the same holds over $S$, which gives the last assertion.
    \end{proof}

    \begin{corollary}\label{Cor:lengthn}
        Let $N\in\FPinf(S)$ with $n=\pd_{R}(N)<\infty$. Then $N$ admits a finite type $R$-projective resolution
        \[
            \begin{tikzcd}
                0 \arrow[r] & P_{n} \arrow[r,"d_{n}"] & \cdots \arrow[r,"d_{1}"] & P_{0} \arrow[r,"\varepsilon"] & N \arrow[r] & 0
            \end{tikzcd}
        \]
        of length $n$, in which $P_{0},\ldots,P_{n-1}$ may be taken to be free.
    \end{corollary}
    \begin{proof}
        By Lemma~\ref{Lem:fpinfquo} we have $N\in\FPinf(R)$, so $N$ has a finite type $R$-free resolution $\bul{Q}$, as noted after Definition~\ref{Def:FPn}. Then the $n$-th syzygy module $\Omega^{n}=\im(d^{Q}_{n})$ is finitely generated and projective since $\pd_{R}(N)=n$, and hence there is an exact sequence
        \[0\longrightarrow\Omega^{n}\longrightarrow Q_{n-1}\longrightarrow\cdots\longrightarrow Q_{0}\longrightarrow N\longrightarrow0\]
        as required.
    \end{proof}

    \section{The chain map \texorpdfstring{$\Phi$}{Φ}}\label{S:maps}
    Throughout this section $N\in\FPinf(S)$ and $\bul{P}=(P_{i},d_{i})_{0\le i\le n}$ is a finite type $R$-projective resolution of $N$ of length $n<\infty$ (for instance $n=\pd_{R}(N)$, which is available by Corollary~\ref{Cor:lengthn}). By Lemma~\ref{Rem:resquo} the chain map $\lambda_{x}:\bul{P}\to\bul{P}$ is null-homotopic, and the reduction modulo $x$ of a null-homotopy for $\lambda_{x}$ is the chain map we need.
    \begin{lemma}\label{Lem:shamash}
        Let $\{\varphi_{i}:P_{i-1}\to P_{i}\}_{i\ge 1}$ be a null-homotopy for the chain map
        $\lambda_{x}:\bul{P}\to\bul{P}$, so that
        \begin{equation}\label{Eq:homotopy}
            d_{i}\varphi_{i}+\varphi_{i+1}d_{i+1}=x\cdot\id_{P_{i}}\qquad\textup{for all }i\ge0,
        \end{equation}
        where $P_{-1}=0$ and $\varphi_{0}=0$. Then the collection $\{\oph_{i}\coloneqq S\otimes_{R}\varphi_{i}:\opp_{i-1}\to\opp_{i}\}_{i\ge1}$ defines a chain map $\Phi:\Sigma\Pol\to\Pol$, and the induced map on first homology groups
        \[\Phi_{\ast}:\Ho_{1}(\Sigma\Pol)\longrightarrow\Ho_{1}(\Pol)\]
        is an isomorphism.
    \end{lemma}
    \begin{proof}
        By reducing \eqref{Eq:homotopy} modulo $x$ we obtain $\od_{i}\oph_{i}+\oph_{i+1}\od_{i+1}=0$, that is
        \begin{equation}\label{Eq:anticomm}
            \oph_{i+1}\od_{i+1}=-\od_{i}\oph_{i}\qquad\text{for all }i\ge0,
        \end{equation}
        which is the condition for the collection $\Phi\coloneqq \{\oph_{i}\}_{i\ge1}$ to be a chain map $\Sigma\Pol\to\Pol$.

        For the second assertion, we have $\Ho_{1}(\Sigma\Pol)\cong\Ho_{0}(\Pol)=\opp_{0}/\im(\od_{1})$, and in this degree $\Phi$ acts by $\oph_{1}$. Taking $i=0$ in \eqref{Eq:homotopy} gives
        \begin{equation}\label{Eq:phi1d1}
            \varphi_{1}d_{1}=x\cdot\id_{P_{0}},
        \end{equation}
        which we will use to identify the induced map $\Phi_{\ast}$.

        Let $\delta:\Ho_{1}(\Pol)\to\ker(\lambda_{x})=N$ be the isomorphism of Lemma~\ref{Lem:tori} and let $u\in P_{0}$ represent a class $[\overline{u}]\in\Ho_{0}(\Pol)$. Then $u\varphi_{1}\in P_{1}$ is a lift of $\overline{u}\,\oph_{1}$, and by \eqref{Eq:phi1d1} we have
        \[(u\varphi_{1})d_{1}=u(\varphi_{1}d_{1})=x\cdot u.\]
        Hence by Remark~\ref{Rem:connecting} we have $\delta(\Phi_{\ast}[\overline{u}])=\varepsilon(u)$, so that the composition $\Phi_{\ast}\delta$ is the canonical isomorphism $\Ho_{0}(\Pol)\cong N$. Then since $\delta$ is an isomorphism, the map $\Phi_{\ast}=\left(\Phi_{\ast}\delta\right)\delta^{-1}$ is an isomorphism, as required.
    \end{proof}
    The construction of $\Phi$ is summarized by the following diagram.
    \[
        \begin{tikzcd}[ampersand replacement=\&, row sep=large, column sep=2.8em]
        	0 \arrow[r] \& P_{n} \arrow[r, "d_{n}"] \arrow[d, "\lambda_{x}"] \& P_{n-1} \arrow[r, "d_{n-1}"] \arrow[dl, "\textcolor{red}{\varphi_{n}}", dashed, red] \arrow[d, "\lambda_{x}"] \& \cdots \arrow[r, "d_{3}"] \& P_{2} \arrow[r,"d_{2}"] \arrow[d, "\lambda_{x}"] \& P_{1} \arrow[r, "d_{1}"] \arrow[dl, "\textcolor{red}{\varphi_{2}}", dashed, red] \arrow[d, "\lambda_{x}"] \& P_{0} \arrow[dl, "\textcolor{red}{\varphi_{1}}", dashed, red] \arrow[d, "\lambda_{x}"] \arrow[r, "\varepsilon"]  \& N \arrow[d, "\lambda_{x}"] \arrow[r] \& 0 \\
        	0 \arrow[r] \& P_{n} \arrow[r, "d_{n}"] \& P_{n-1} \arrow[r, "d_{n-1}"] \& \cdots \arrow[r, "d_{3}"]  \& P_{2} \arrow[r,"d_{2}"] \& P_{1} \arrow[r, "d_{1}"] \& P_{0} \arrow[r, "\varepsilon"]  \& N \arrow[r] \& 0
        \end{tikzcd}
    \]
    The dashed arrows form the null-homotopy, and after applying the functor $S\otimes_{R}-$ they determine the chain map $\Phi:\Sigma\Pol\to\Pol$.

    \begin{remark}\label{Rem:shamashprop1}
        In \cite[Lem.~1, p.454]{shamash1} the maps $\{\varphi_{i}\}_{i\ge1}$ are constructed for a minimal free resolution over a commutative local ring with $x\in\ma\setminus\ma^{2}$. However, this construction only uses that $x$ annihilates the module being resolved (see \cite[p.455]{shamash1}), which is the content of Lemma~\ref{Rem:resquo} above. In this local setting the maps $\{\varphi_{i}\}_{i\ge1}$ may be chosen so that $\im(\varphi_{i})$ is a direct summand of $P_{i}$ and $\ker(\varphi_{i+1})=\im(\varphi_{i})$ for all $i\ge1$ \cite[Prop.~1, p.455]{shamash1}. In particular $\varphi_{i}\varphi_{i+1}=0$ for all $i\ge1$, and this motivates the hypothesis $\Phi^{2}=0$ which we impose in the next section.
    \end{remark}

    \section{General construction}\label{S:gencons}
    In this section we build the resolution of Theorem~\ref{thm:1} as the limit of a sequence of iterated mapping cones. Throughout, $0\neq N\in\FPinf(S)$ and $\bul{P}\to N$ is a finite type $R$-projective resolution of length $n<\infty$ (for instance $n=\pd_{R}(N)$, which is available by Corollary~\ref{Cor:lengthn}), equipped with a fixed null-homotopy $\{\varphi_{i}\}_{i\ge1}$ for $\lambda_{x}:\bul{P}\to\bul{P}$ as in Lemma~\ref{Lem:shamash}. The reduction of this null-homotopy modulo $x$ is the chain map $\Phi:\Sigma\Pol\to\Pol$. Recall that compositions are written in the order of application, as in Section~\ref{S:notation}, so that $\varphi_{i}\varphi_{i+1}$ means $\varphi_{i}$ followed by $\varphi_{i+1}$.

    We follow the iterated mapping cone construction of Nguyen and Veliche \cite[\S3]{imc}. Set $\bul{C}^{0}\coloneqq\Pol$ and, for $i\ge1$, we define
    \[\bul{C}^{i}\coloneqq\Sigma^{i}\Pol,\qquad \mu^{i}\coloneqq\Sigma^{i-1}\Phi:\bul{C}^{i}\longrightarrow\bul{C}^{i-1}.\]
    By Lemma~\ref{Lem:tori} each $\bul{C}^{i}$ is a bounded-below complex of finitely generated projective $S$-modules whose homology is concentrated in degrees $i$ and $i+1$, where it is $N$. Moreover, since $\mu^{i}$ is a suspension of $\Phi$, it follows from Lemma~\ref{Lem:shamash} that $\mu^{i}$ induces an isomorphism $\Ho_{i}(\bul{C}^{i})\to\Ho_{i}(\bul{C}^{i-1})$ for all $i\ge1$.

    The idea is to take mapping cones along the maps $\mu^{i}$ one after another. Firstly, consider the mapping cone $\bul{M}^{1}\coloneqq\Con(\Phi)$. Since $\Phi_{\ast}$ is an isomorphism, the long exact sequence of this mapping cone shows that the homology of $\bul{M}^{1}$ is $N$ in degrees $0$ and $3$. In this way the homology of $\Pol$ in degree $1$ has been moved up to degree $3$, and each further mapping cone moves it up by two more degrees, so that in the limit it disappears.

    To continue, at the $i$-th stage of this process, for $i\ge1$, we need a chain map
    \[\psi^{i+1}:\Sigma^{i}\bul{C}^{i+1}\longrightarrow\bul{M}^{i},\]
    and we then set $\bul{M}^{i+1}\coloneqq\Con(\psi^{i+1})$. Since $\bul{M}^{i}=\Con(\psi^{i})$ is itself a mapping cone, it contains $\Sigma^{i}\bul{C}^{i}$ as a graded direct summand, and following Nguyen and Veliche we take $\psi^{i+1}$ to be the composition
    \[\begin{tikzcd}[column sep=huge]
        {\Sigma^{i}\bul{C}^{i+1}} \arrow[r, "{\Sigma^{i}\mu^{i+1}}"] & {\Sigma^{i}\bul{C}^{i}} \arrow[r, "{\iota^{i}}"] & {\bul{M}^{i}},
    \end{tikzcd}\]
    where $\iota^{i}$ denotes the inclusion of this summand. This composition is a map of graded modules, and in Lemma~\ref{Lem:psidef} we show that it is a chain map under the assumption that $\mu^{i+1}\mu^{i}=0$. Nguyen and Veliche call a sequence of complexes $\{\bul{C}^{i}\}_{i\ge0}$ with maps $\{\mu^{i}\}_{i\ge1}$ satisfying this property a \emph{complex of complexes} \cite[Prop.~3.3]{imc}.

    In degree $j$ the composition $\mu^{i+1}\mu^{i}:C^{i+1}_{j}\to C^{i-1}_{j}$ is, up to sign, given by
    \[
        \begin{tikzcd}[sep=large]
            \opp_{j-i-1} \arrow[r, "{\oph_{j-i}}"] & \opp_{j-i} \arrow[r, "{\oph_{j-i+1}}"] & \opp_{j-i+1},
        \end{tikzcd}
    \]
    so that this condition is equivalent to $\oph_{a}\oph_{a+1}=0$ for all $a\ge1$. Regarding the collection $\{\oph_{i}\}_{i\ge1}$ as a single endomorphism $\oph$ of the graded module $\opp=\bigoplus_{i}\opp_{i}$, which maps $\opp_{i-1}$ to $\opp_{i}$ for all $i\ge1$, we can write this condition as $\oph^{2}=0$, which we abbreviate to $\Phi^{2}=0$.

    \begin{hypothesis}\label{Hyp:star}
        The reduced null-homotopy satisfies $\Phi^{2}=0$, that is,
        \[\oph_{i}\oph_{i+1}=0:\opp_{i-1}\longrightarrow\opp_{i+1}\qquad\text{for all }i\ge1.\]
    \end{hypothesis}

    Over $R$ the hypothesis is equivalent to $\im(\varphi_{i}\varphi_{i+1})\subseteq xP_{i+1}$ for all $i\ge1$. By Remark~\ref{Rem:shamashprop1} it holds for the null-homotopies of Shamash in the local setting, and by Proposition~\ref{Prop:starlow} it holds automatically when $n\le2$ and reduces to a single equation when $n=3$. We discuss the hypothesis further in Section~\ref{S:whenstar}, and we assume that it holds for the rest of this section and throughout Section~\ref{S:period}.

    For all $k$ we write $\iota^{i}_{k-i}$ for the component of $\iota^{i}$ which maps $C^{i}_{k-i}=(\Sigma^{i}\bul{C}^{i})_{k}$ to $M^{i}_{k}$, so that the components of $\iota^{i}$ are indexed by the degree in $\bul{C}^{i}$, and those of $\psi^{i+1}$ by the degree in $\bul{M}^{i}$.
    \begin{lemma}[cf. {\cite[Prop.~3.3]{imc}}]\label{Lem:psidef}
        Assume Hypothesis~\ref{Hyp:star}. Set $\bul{M}^{0}\coloneqq\bul{C}^{0}$ and, for $i\ge0$, let $\psi^{i+1}:\Sigma^{i}\bul{C}^{i+1}\to\bul{M}^{i}$ be the graded map of degree $0$ whose component in degree $k$ is
        \begin{equation}\label{Eq:psidef}
            \psi^{i+1}_{k}\coloneqq
            \Bigl(C^{i+1}_{k-i}\xrightarrow{\ \mu^{i+1}_{k-i}\ }C^{i}_{k-i}
            \xrightarrow{\ (-1)^{i}\iota^{i}_{k-i}\ } M^{i}_{k}\Bigr),
            \qquad 
            \bul{M}^{i+1}\coloneqq\Con(\psi^{i+1}).
        \end{equation}
        Then $\psi^{i+1}$ is a chain map for all $i\ge0$, and hence $\bul{M}^{i+1}\coloneqq\Con(\psi^{i+1})$ is defined for all $i\ge0$. Moreover $\psi^{1}=\Phi$, and for all $i\ge1$ the triangle
        \begin{equation}\label{Eq:triangle}
            \begin{tikzcd}[row sep=large, column sep=large ]
            	&&& {\Sigma^{i} \bul{C}^{i+1}} \\
            	0 & {\bul{M}^{i-1}} & {\bul{M}^{i}} & {\Sigma^{i} \bul{C}^{i}} & 0
            	\arrow["{\psi^{i+1}}"', from=1-4, to=2-3]
            	\arrow["{\Sigma^{i} \mu^{i+1}}", from=1-4, to=2-4]
            	\arrow[from=2-1, to=2-2]
            	\arrow["{\eta^{i-1}}", from=2-2, to=2-3]
            	\arrow["{\nu^{i}}", from=2-3, to=2-4]
            	\arrow[from=2-4, to=2-5]
            \end{tikzcd}
        \end{equation}
        commutes, its row being the short exact sequence of the mapping cone $\bul{M}^{i}$.
    \end{lemma}
    \begin{proof}
        We proceed by induction on $i\ge0$. For $i=0$ we have $\bul{M}^{0}=\bul{C}^{0}$ and $\iota^{0}=\id$, so that $\psi^{1}=\mu^{1}=\Phi$ is a chain map by Lemma~\ref{Lem:shamash}.

        Now let $i\ge1$ and suppose that $\psi^{i}$ is a chain map, so that $\bul{M}^{i}=\Con(\psi^{i})$ is defined. Let $\iota^{i}:\Sigma^{i}\bul{C}^{i}\to\bul{M}^{i}$ denote the inclusion of the summand $\Sigma^{i}\bul{C}^{i}$ of the mapping cone $\bul{M}^{i}$, and let $\eta^{i-1}:\bul{M}^{i-1}\to\bul{M}^{i}$ denote the inclusion in the short exact sequence of this mapping cone. In this notation the definition of $\psi^{i+1}$ reads
        \begin{equation}\label{Eq:indhyp}
            \psi^{i+1}=\left(\Sigma^{i}\mu^{i+1}\right)\iota^{i}.
        \end{equation}
        Since $\Sigma^{i}\bul{C}^{i}=\Sigma\left(\Sigma^{i-1}\bul{C}^{i}\right)$ as graded modules, we may regard $\psi^{i}$ as a graded map $\Sigma^{i}\bul{C}^{i}\to\bul{M}^{i-1}$ of degree $-1$, and we do so from now on.

        The inclusion $\iota^{i}$ is a map of graded modules but not a chain map, and its failure to commute with the differentials is given by the first row of the differential of the mapping cone, namely
        \[
            \iota^{i}\partial^{M^{i}}-\partial^{\Sigma^{i}C^{i}}\iota^{i}
            =-\psi^{i}\eta^{i-1}.
        \]
        Precomposing with the chain map $\Sigma^{i}\mu^{i+1}$, and using \eqref{Eq:indhyp} both for $\psi^{i+1}$ and for $\psi^{i}$, we obtain
        \[
            \psi^{i+1}\partial^{M^{i}} - \partial^{\Sigma^{i}C^{i+1}}\psi^{i+1}
            = -\left(\Sigma^{i}\mu^{i+1}\right)\left(\Sigma^{i-1}\mu^{i}\right)\iota^{i-1}\eta^{i-1}.
        \]
        In degree $k$ the two suspended maps compose to
        \[
            \left(\Sigma^{i}\mu^{i+1}\right)_{k}\left(\Sigma^{i-1}\mu^{i}\right)_{k-1}
            =-\left(\mu^{i+1}\mu^{i}\right)_{k-i}:C^{i+1}_{k-i}\longrightarrow C^{i-1}_{k-i},
        \]
        which is zero by Hypothesis~\ref{Hyp:star}. Hence $\psi^{i+1}$ is a chain map and $\bul{M}^{i+1}$ is defined.

        Finally, to see that the triangle \eqref{Eq:triangle} commutes, note that $\nu^{i}$ is the projection of $\bul{M}^{i}$ onto the summand $\Sigma^{i}\bul{C}^{i}$, so that $\iota^{i}\nu^{i}=\id$. Then by \eqref{Eq:indhyp} we have $\psi^{i+1}\nu^{i}=(\Sigma^{i}\mu^{i+1})\iota^{i}\nu^{i}=\Sigma^{i}\mu^{i+1}$, as required.
    \end{proof}

    In this way the mapping cones form a directed sequence $\{\bul{M}^{i},\eta^{i}\}_{i\ge 0}$ of complexes of finitely generated projective $S$-modules with injective maps $\eta^{i}:\bul{M}^{i}\to \bul{M}^{i+1}$, and we write $\bul{M}\coloneqq\dirlim_{i}\bul{M}^{i}$ for its direct limit.

    \begin{remark}\label{Rem:reindex}
        Since $C^{i+1}_{j}=(\Sigma^{i+1}\Pol)_{j}=\opp_{j-i-1}$, the only possibly non-zero component of $\psi^{i+1}$ in degree $k$ is $\oph_{k-2i}:\opp_{k-2i-1}\to\opp_{k-2i}$, and in particular $\psi^{i+1}$ vanishes in degrees $k\le 2i$. We re-index accordingly, writing
        \[\psi_{j}:\opp_{j-1}\longrightarrow M^{i}_{2i+j}\qquad(1\le j\le n)\]
        for the component of $\psi^{i+1}$ in degree $2i+j$. In these coordinates $\psi_{j}$ is the map $\oph_{j}$ followed by the inclusion of $\opp_{j}$ as a direct summand of $M^{i}_{2i+j}$. In particular the superscript may be suppressed and $\psi_{j}$ identified with $\oph_{j}:\opp_{j-1}\to\opp_{j}$, which is the convention used in all of the matrices below.
    \end{remark}

    At each stage the long exact sequence in homology of the mapping cone has the same form, and we record it in the following lemma (cf. \cite[Rem.~3.4]{imc}).
    \begin{lemma}\label{Lem:les}
        Fix $i\ge0$, write $\bul{M}^{i+1}=\Con(\psi^{i+1})$ and suppose that
        \begin{equation}\label{Eq:homi}
            \Ho_{j}(\bul{M}^{i})=\begin{cases}N & j=0,\;2i+1,\\ 0 &\textup{else.}\end{cases}
        \end{equation}
        Then the following hold.
        \begin{enumerate}[label=\textup{(\arabic*)}]
            \item The induced map $\nu^{i+1}_{\ast}$ in degree $2i+3$ and the induced map $\eta^{i}_{\ast}$ in degree $0$ are isomorphisms, so that $\Ho_{2i+3}(\bul{M}^{i+1})\cong N$ and $\Ho_{0}(\bul{M}^{i+1})\cong N$.

            \item If $\psi^{i+1}_{\ast}$ is an isomorphism in degree $2i+1$, then
            \[\Ho_{j}(\bul{M}^{i+1})=\begin{cases}N & j=0,\;2i+3,\\ 0 &\textup{else,}\end{cases}\]
            which is the same description of the homology with $i$ replaced by $i+1$.
        \end{enumerate}
    \end{lemma}
    \begin{proof}
        Consider the short exact sequence of the mapping cone $\bul{M}^{i+1}$. Since $\psi^{i+1}$ maps $\Sigma^{i}\bul{C}^{i+1}=\Sigma^{2i+1}\Pol$ to $\bul{M}^{i}$, this sequence is of the form:
        \[
        \begin{tikzcd}
            0 \arrow[r] & \bul{M}^{i} \arrow[r,"\eta^{i}"] & \bul{M}^{i+1} \arrow[r,"\nu^{i+1}"] & \Sigma^{2i+2}\Pol \arrow[r] & 0.
        \end{tikzcd}
        \]
        By Lemma~\ref{Lem:tori} the homology of $\Sigma^{2i+2}\Pol$ is concentrated in degrees $2i+2$ and $2i+3$, where it is $N$. Then since the homology of $\bul{M}^{i}$ and of $\Sigma^{2i+2}\Pol$ vanishes outside the degrees $0$, $2i+1$, $2i+2$ and $2i+3$, the associated long exact sequence collapses into
        \begin{gather*}
            0\longrightarrow\Ho_{2i+3}(\bul{M}^{i+1})\xrightarrow{\ \nu^{i+1}_{\ast}\ }\Ho_{2i+3}(\Sigma^{2i+2}\Pol)\longrightarrow0, \\
            0\to \Ho_{2i+2}(\bul{M}^{i+1})\to\Ho_{2i+2}(\Sigma^{2i+2}\Pol)\xrightarrow{\ -\psi^{i+1}_{\ast}\ }\Ho_{2i+1}(\bul{M}^{i})\to\Ho_{2i+1}(\bul{M}^{i+1})\to0, \\
            0\longrightarrow\Ho_{0}(\bul{M}^{i})\xrightarrow{\ \eta^{i}_{\ast}\ }\Ho_{0}(\bul{M}^{i+1})\longrightarrow0,
        \end{gather*}
        with $\Ho_{j}(\bul{M}^{i+1})=0$ in every remaining degree. The map in the middle sequence is the connecting homomorphism of this short exact sequence, where we identify $\Ho_{2i+2}(\Sigma^{2i+2}\Pol)$ with $\Ho_{2i+1}(\Sigma^{2i+1}\Pol)$. The first and last of the three sequences give part (1). If $\psi^{i+1}_{\ast}$ is an isomorphism, then the middle sequence gives $\Ho_{2i+2}(\bul{M}^{i+1})=\Ho_{2i+1}(\bul{M}^{i+1})=0$, which is part (2).
    \end{proof}

    \begin{theorem}\label{Thm:imc}
        Assume Hypothesis~\ref{Hyp:star}, and let $\{\bul{M}^{i},\eta^{i}\}_{i\ge 0}$ be the sequence of mapping cones of Lemma~\ref{Lem:psidef}. Then the following hold:
        \begin{enumerate}[label=\textup{(\arabic*)}]
            \item For $i>0$ the $i$-th mapping cone $\bul{M}^{i}=(M^{i}_{j},\partial^{i}_{j})$ is given by
            \[M^{i}_{j}=\opp_{j-2i}\oplus M^{i-1}_{j},
            \qquad
            \partial^{i}_{j}=\renewcommand\arraystretch{1.4}\begin{bmatrix}
                \od_{j-2i} & -\psi_{j-2i+1}\\
                0 & \partial^{i-1}_{j}
            \end{bmatrix}
            \qquad(j\in\Z),\]
            where $\psi_{j-2i+1}$ is $\oph_{j-2i+1}$ followed by the injection of $\opp_{j-2i+1}$ into $M^{i-1}_{j-1}$, as in Remark~\ref{Rem:reindex}. In particular $\bul{M}^{i}$ is a complex of finitely generated projective $S$-modules.
            
            \item For $i\ge0$ the homology of the $i$-th mapping cone $\bul{M}^{i}$ is given by
            \[\mathrm{H}_{j}(\bul{M}^{i})=\begin{cases}
                N & j=0,\;2i+1,\\
                0 &\textup{else}.
            \end{cases}\]
            
            \item For $i\ge0$ the maps $\psi^{i+1}:\Sigma^{i} \bul{C}^{i+1}\to \bul{M}^{i}$ induce isomorphisms on homology
            \[\psi^{i+1}_{\ast}:\mathrm{H}_{2i+1}(\Sigma^{i} \bul{C}^{i+1})\to \mathrm{H}_{2i+1}(\bul{M}^{i}),\]
            or equivalently $\psi^{i+1}_{*}:\mathrm{H}_{2i+1}(\Sigma^{2i+1}\Pol) \to \mathrm{H}_{2i+1}(\bul{M}^{i})$.
            
            \item The direct limit $\bul{M}=\dirlim_{i} \bul{M}^{i}$ is a finite type $S$-projective resolution of $N$.
        \end{enumerate}
    \end{theorem}
    \begin{proof}
        We prove parts (1) to (3) together by induction on $i\ge0$, and then part (4) follows. For $i=0$ we have $\bul{M}^{0}=\Pol$, so part (2) is Lemma~\ref{Lem:tori} and part (3) is Lemma~\ref{Lem:shamash}.

        Now let $i\ge0$ and suppose that parts (2) and (3) hold for $\bul{M}^{i}$. Then by Remark~\ref{Rem:reindex} the component of $\psi^{i+1}$ in degree $2i+j$ is $\oph_{j}$ followed by the inclusion of $\opp_{j}$ into $M^{i}_{2i+j}$, and hence the mapping cone of $\psi^{i+1}$ is given by part (1) with $i$ replaced by $i+1$. The two parts of Lemma~\ref{Lem:les} then give part (2) for $\bul{M}^{i+1}$, together with an isomorphism $\nu^{i+1}_{\ast}$ in degree $2i+3$. For part (3), the commutative triangle \eqref{Eq:triangle} at stage $i+1$ gives $\psi^{i+2}\nu^{i+1}=\Sigma^{2i+2}\Phi$, and $\Sigma^{2i+2}\Phi$ is an isomorphism on $\Ho_{2i+3}$ by Lemma~\ref{Lem:shamash}. Hence $\psi^{i+2}_{\ast}=(\Sigma^{2i+2}\Phi)_{\ast}(\nu^{i+1}_{\ast})^{-1}$ is an isomorphism in degree $2i+3$, and the induction is complete.

        For part (4), the direct limit $\bul{M}=\dirlim_{i} \bul{M}^{i}$ is the union $\bigcup_{i\ge0}\bul{M}^{i}$ by \cite[Ex.~3.2.9, p.109]{dercat}, and since this colimit is filtered it commutes with homology \cite[Cor.~2.6.17, p.58]{weibel}. Then since $\Ho_{j}(\bul{M}^{i})=0$ for $0<j\le 2i$, and the induced maps $\eta^{i}_{\ast}$ are isomorphisms in degree $0$ by Lemma~\ref{Lem:les}(1), we have
        \[\mathrm{H}_{j}(\bul{M})=\mathrm{H}_{j}\left(\dirlim_{i}\bul{M}^{i}\right)=\dirlim_{i}\mathrm{H}_{j}(\bul{M}^{i}) = \begin{cases}
    		N & \text{if } j=0, \\
    		0 & \text{else.}\\ 
    	\end{cases}\]
        Moreover, it follows from part (1) that $M^{i}_{j}=M^{i-1}_{j}$ whenever $2i>j$, so that $M_{j}=M^{i}_{j}$ for all $i\ge\floor*{\frac{j}{2}}$, and hence each $M_{j}$ is a finitely generated projective $S$-module. Then $\bul{M}\to N$ is a finite type $S$-projective resolution of $N$, as required.
    \end{proof}
    
    \section{Eventual periodicity}\label{S:period}
    In this section we study the structure of the sequence of complexes $\{\bul{M}^{i},\eta^{i}\}_{i\ge 0}$ produced in Section~\ref{S:gencons}, with the aim of showing that their limit $\bul{M}$ is eventually $2$-periodic. Since the maps $\psi^{i+1}$ are given explicitly in Lemma~\ref{Lem:psidef}, the recursion of Theorem~\ref{Thm:imc} can be solved in closed form, and everything below is read off from that solution. Throughout this section we assume Hypothesis~\ref{Hyp:star}.

    \begin{lemma}\label{Lem:closedform}
        For all $i\ge0$ and all $j\in\Z$, we have
        \[M^{i}_{j}=\bigoplus_{k=0}^{i}\opp_{j-2k},
        \qquad\textup{and}\qquad
        M_{j}=\bigoplus_{k\ge 0}\opp_{j-2k}=\bigoplus_{\substack{0\le a\le\min(j,n)\\ a\equiv j\!\!\pmod{2}}}\opp_{a},\]
        where $\opp_{a}=0$ for $a<0$ and for $a>n$. We call $k$ the \emph{row} of the summand $\opp_{j-2k}$ of $M^{i}_{j}$ and of $M_{j}$. In these coordinates the differential $\partial_{j}:M_{j}\to M_{j-1}$ has component $\opp_{a}\to\opp_{b}$ equal to
        \[\begin{cases}
            \od_{a} & \textup{if } b=a-1,\\
            -\oph_{a+1} & \textup{if } b=a+1,\\
            0 & \textup{otherwise},
        \end{cases}\]
        and $\partial^{i}_{j}$ is the restriction of $\partial_{j}$ to the summands of $M_{j}$ and of $M_{j-1}$ in rows $k\le i$.
    \end{lemma}
    \begin{proof}
        Firstly, by Theorem~\ref{Thm:imc}(1) we have $M^{0}_{j}=\opp_{j}$, and $M^{i}_{j}=\opp_{j-2i}\oplus M^{i-1}_{j}$ for $i>0$. Then induction on $i\ge 0$ gives $M^{i}_{j}=\bigoplus_{k=0}^{i}\opp_{j-2k}$. The summand in row $k$ vanishes when $2k>j$, and hence the union over $i\ge 0$ is the finite sum $M_{j}=\bigoplus_{k\ge0}\opp_{j-2k}$.

        Now set $a=j-2k$. Then $a\equiv j\pmod{2}$, and the condition $k\ge0$ becomes $a\le j$. Moreover $\opp_{a}=0$ unless $0\le a\le n$, which gives the second description of $M_{j}$. Since $a$ and $j$ have the same parity, the condition $a\le n$ implies $a\le j$ when $j\ge n-1$, and hence the condition $a\le j$ may be omitted in this case.

        For the differentials, the description in Theorem~\ref{Thm:imc}(1) exhibits $\partial^{i}_{j}$ as a $2\times2$ block matrix. The diagonal blocks of this matrix are $\od_{j-2i}$ on the summand in row $i$ and $\partial^{i-1}_{j}$ on rows $k<i$, and its only non-zero off-diagonal block is $-\psi_{j-2i+1}$. By Remark~\ref{Rem:reindex} this block maps the summand $\opp_{a}$ with $a=j-2i$ to the summand $\opp_{a+1}$ of $M^{i-1}_{j-1}$ by $-\oph_{a+1}$. Then by reading off the two blocks in the coordinates $a$ we obtain the components displayed, and by induction on $i\ge0$ this holds for all rows. Finally, for each $i\ge 0$, the chain map $\eta^{i}:\bul{M}^{i}\to\bul{M}^{i+1}$ is just the inclusion of rows $k\le i$, which gives the last assertion.
    \end{proof}
    In a slightly more informal sense, the complex $\bul{M}$ can also be realized via the following `staircase diagram' of copies of $\Pol$, in which $\psi$ denotes the relevant component of the chain map $\Phi$ as in Remark~\ref{Rem:reindex}. 
        \[\begin{tikzcd}[ampersand replacement=\&, row sep = scriptsize, column sep = small, cramped,
            execute at end picture={%
            \node[draw=red, dashed, thick, rounded corners=5pt, inner sep=5pt, rotate fit=45, fit=(dA)(dA2)(dA3)] (diag) {}; \node[red, anchor=south west, inner sep=1pt] at (diag.east) {$M_{4}$};}]
			{\opp_{n}} \& {\opp_{n-1}} \& {\opp_{n-2}} \& {\opp_{n-3}} \& {\opp_{n-4}} \& {\cdots} \& {\opp_{1}} \& {\opp_{0}} \\
			\& {\opp_{n}} \& {\opp_{n-1}} \& {\opp_{n-2}} \& {\opp_{n-3}} \& {\cdots} \& {\opp_{2}} \& {\opp_{1}} \&  {\opp_{0}} \\
			\&\&  {\opp_{n}} \&  {\opp_{n-1}} \&  {\opp_{n-2}} \&  {\cdots} \&  {\opp_{3}} \&  {\opp_{2}} \&  {\opp_{1}} \& |[alias=dA3]| {\opp_{0}} \\
			\&\&\& {\opp_{n}} \& {\opp_{n-1}} \& \cdots \& {\opp_{4}} \& {\opp_{3}} \& |[alias=dA2]| {\opp_{2}} \& {\opp_{1}} \& |[alias=dB2]| {\opp_{0}} \\
			\&\&\&\& {\opp_{n}} \& \cdots \& {\opp_{5}} \& |[alias=dA]| {\opp_{4}} \& {\opp_{3}} \& |[alias=dB]| {\opp_{2}} \& {\opp_{1}} \& {\opp_{0}}
			\arrow["{\od}", from=1-1, to=1-2]
			\arrow["{\od}", from=1-2, to=1-3]
			\arrow["{-\psi}", from=1-2, to=2-2]
			\arrow["{\od}", from=1-3, to=1-4]
			\arrow["{-\psi}", from=1-3, to=2-3]
			\arrow["{\od}", from=1-4, to=1-5]
			\arrow["{-\psi}", from=1-4, to=2-4]
			\arrow[from=1-5, to=1-6]
			\arrow["{-\psi}", from=1-5, to=2-5]
			\arrow[from=1-6, to=1-7]
			\arrow["{\od}", from=1-7, to=1-8]
			\arrow["{-\psi}", from=1-7, to=2-7]
			\arrow["{-\psi}", from=1-8, to=2-8]
			\arrow["{\od}", from=2-2, to=2-3]
			\arrow["{\od}", from=2-3, to=2-4]
			\arrow["{-\psi}", from=2-3, to=3-3]
			\arrow["{\od}", from=2-4, to=2-5]
			\arrow["{-\psi}", from=2-4, to=3-4]
			\arrow[from=2-5, to=2-6]
			\arrow["{-\psi}", from=2-5, to=3-5]
			\arrow[from=2-6, to=2-7]
			\arrow["{\od}", from=2-7, to=2-8]
			\arrow["{-\psi}", from=2-7, to=3-7]
			\arrow["{\od}", from=2-8, to=2-9]
			\arrow["{-\psi}", from=2-8, to=3-8]
			\arrow["{-\psi}", from=2-9, to=3-9]
			\arrow["{\od}", from=3-3, to=3-4]
			\arrow["{\od}", from=3-4, to=3-5]
			\arrow["{-\psi}", from=3-4, to=4-4]
			\arrow[from=3-5, to=3-6]
			\arrow["{-\psi}", from=3-5, to=4-5]
			\arrow[from=3-6, to=3-7]
			\arrow["{\od}", from=3-7, to=3-8]
			\arrow["{-\psi}", from=3-7, to=4-7]
			\arrow["{\od}", from=3-8, to=3-9]
			\arrow["{-\psi}", from=3-8, to=4-8]
			\arrow["{\od}", from=3-9, to=3-10]
			\arrow["{-\psi}", from=3-9, to=4-9]
			\arrow["{-\psi}", from=3-10, to=4-10]
			\arrow["{\od}", from=4-4, to=4-5]
			\arrow[from=4-5, to=4-6]
			\arrow["{-\psi}", from=4-5, to=5-5]
			\arrow[from=4-6, to=4-7]
			\arrow["{\od}", from=4-7, to=4-8]
			\arrow["{-\psi}", from=4-7, to=5-7]
			\arrow["{\od}", from=4-8, to=4-9]
			\arrow["{-\psi}", from=4-8, to=5-8]
			\arrow["{\od}", from=4-9, to=4-10]
			\arrow["{-\psi}", from=4-9, to=5-9]
			\arrow["{\od}", from=4-10, to=4-11]
			\arrow["{-\psi}", from=4-10, to=5-10]
			\arrow["{-\psi}", from=4-11, to=5-11]
			\arrow[from=5-5, to=5-6]
			\arrow[from=5-6, to=5-7]
			\arrow["{\od}", from=5-7, to=5-8]
			\arrow["{\od}", from=5-8, to=5-9]
			\arrow["{\od}", from=5-9, to=5-10]
			\arrow["{\od}", from=5-10, to=5-11]
			\arrow["{\od}", from=5-11, to=5-12]
		\end{tikzcd}\]
    The bottom row is row $0$, which is $\bul{M}^{0}=\Pol$. For $i\ge1$, row $i$ is the copy of $\Sigma^{2i}\Pol$ contributed by the $i$-th mapping cone, and it is shifted one column to the left of row $i-1$. The degree drops by one on moving right along a row, and also on moving down a column. It follows that the summands of a fixed degree $j$ lie along a diagonal running upwards and to the right starting in row $0$. The module $M_{j}$ is the direct sum along that diagonal, and the differential $\partial_{j}$ is the collection of maps out of it. For example, the diagonal outlined in red is $M_{4}=\opp_{4}\oplus\opp_{2}\oplus\opp_{0}$.
    
    The closed form of Lemma~\ref{Lem:closedform} has the following consequences, where as before $n<\infty$ is the length of $\bul{P}$.
    \begin{corollary}\label{Cor:stabper}
        Let $\bul{M}=\dirlim_{i} \bul{M}^{i}$. Then the following hold.
        \begin{enumerate}[label=\textup{(\arabic*)}]
            \item For all $j\ge0$ and all $i\ge\floor*{\frac{j}{2}}$ we have $M^{i}_{j}=M_{j}$ and $\partial^{i}_{j}=\partial_{j}$. In particular the whole of $\bul{M}$ is determined by $\bul{M}^{\ceil*{n/2}}$.

            \item $M_{j}=M_{j+2}$ for all $j\ge n-1$, and $\partial_{j}=\partial_{j+2}$ for all $j\ge n$.

            \item For $0\le j\le n-1$ we have 
            \[M_{j}=\bigoplus\limits_{k=0}^{\floor*{\frac{j}{2}}}\opp_{j-2k},\] 
            which is a direct summand of $M_{n-1}$ or of $M_{n}$ according to the parity of $j$, and $\partial_{j}$ is the corresponding submatrix of $\partial_{n+1}$ or of $\partial_{n}$.
        \end{enumerate}
    \end{corollary}
    \begin{proof}
        (1) By Lemma~\ref{Lem:closedform} the summands of $M_{j}$ lie in rows $k\le\floor*{\frac{j}{2}}$, and hence $M^{i}_{j}=M_{j}$ for all $i\ge\floor*{\frac{j}{2}}$. Since $\floor*{\frac{j-1}{2}}\le\floor*{\frac{j}{2}}$, we also have $M^{i}_{j-1}=M_{j-1}$ for these $i$, and hence $\partial^{i}_{j}=\partial_{j}$. In particular, since $\floor*{\frac{n+1}{2}}=\ceil*{\frac{n}{2}}$, the complex $\bul{M}^{\ceil*{n/2}}$ determines $M_{j}$ and $\partial_{j}$ for all $j\le n+1$, and by (2) these determine the whole of $\bul{M}$.

        (2) By Lemma~\ref{Lem:closedform} the module $M_{j}$ depends only on the parity of $j$ for all $j\ge n-1$, and the component of $\partial_{j}$ from $\opp_{a}$ to $\opp_{b}$ depends only on $a$ and $b$. Hence $\partial_{j}=\partial_{j+2}$ when both $j$ and $j-1$ are at least $n-1$, that is, for all $j\ge n$.

        (3) This follows directly from Lemma~\ref{Lem:closedform} and (2).
    \end{proof}

    Together, Lemma~\ref{Lem:closedform} and Corollary~\ref{Cor:stabper} give a complete description of the resolution $\bul{M}$, as follows:
    \begin{theorem}\label{Thm:main}
        Assume Hypothesis~\ref{Hyp:star} and let $\bul{M}=(M_{j},\partial_{j})_{j\ge 0}$ be the limit of the sequence of Theorem~\ref{Thm:imc}. Then $\bul{M}$ is a finite type $S$-projective resolution of $N$ with
        \[M_{j} = \bigoplus_{\substack{0\le a\le\min(j,n)\\ a\equiv j\!\!\pmod{2}}}\opp_{a}\qquad(j\ge0)\]
        and with the components of $\partial_{j}$ as described in Lemma~\ref{Lem:closedform}. Moreover $M_{j}=M_{j+2}$ for all $j\ge n-1$ and $\partial_{j}=\partial_{j+2}$ for all $j\ge n$, so that $\bul{M}$ is $2$-periodic in degrees $\ge n-1$.
    \end{theorem}
    \begin{proof}
        The fact that $\bul{M}$ is a finite type $S$-projective resolution of $N$ is part (4) of Theorem~\ref{Thm:imc}. Similarly, the description of the modules and differentials is Lemma~\ref{Lem:closedform}, and the periodicity is Corollary~\ref{Cor:stabper}(2).
    \end{proof}

    \begin{remark}\label{Rem:sharp}
        By Lemma~\ref{Lem:tori} we have $\Ho_{0}(\Pol)\cong\Ho_{1}(\Pol)\cong N\neq0$, and hence $\opp_{0}$ and $\opp_{1}$ are non-zero. For each $j\ge0$ one of them is the summand of $M_{j}$ in row $\floor*{j/2}$, which does not lie in $M^{i}_{j}$ for $i<\floor*{j/2}$. It follows that the bound in Corollary~\ref{Cor:stabper}(1) cannot be improved. Now suppose that $\opp_{n}\neq0$, that is, $xP_{n}\neq P_{n}$. Then $\opp_{n}$ is a summand of $M_{n}$ but not of $M_{n-2}$, and hence $M_{n-2}\neq M_{n}$. Moreover, $\partial_{n-1}$ maps to $M_{n-2}$ whereas $\partial_{n+1}$ maps to $M_{n}$, so that $\partial_{n-1}\neq\partial_{n+1}$. Hence the two periodicity bounds of Theorem~\ref{Thm:main} cannot be improved either.
    \end{remark}

    If we order the summands of $M_{j}$ by increasing $a$, then the differential has the following matrix form. Firstly, for $j\ge n$ each of the modules $\opp_{a}$ with $a\le n$ occurs either in $M_{j}$ when $a\equiv j\pmod{2}$ or in $M_{j-1}$ otherwise. Moreover, the differentials are given by
    \[
        \partial_{j} = \begin{cases}
            \partial_{d} & j\equiv1\pmod{2},\\
            \partial_{\varphi} & j\equiv0\pmod{2},
        \end{cases}
    \]
    where
    \begin{equation}\label{Eq:matrices}
    \begin{aligned}
        \partial_{d} &= \renewcommand\arraystretch{1.5}\begin{bmatrix}
                \od_{1}    & -\oph_{2} & 0         & 0        & \cdots \\
                0          & \od_{3}   & -\oph_{4} & 0        & \cdots \\
                \vdots     & 0         & \od_{5}   & -\oph_{6}& \ddots \\
                \vdots     & \vdots    & 0         & \od_{7}  & \ddots \\
                \vdots     & \vdots    & \vdots    & \vdots   & \ddots
            \end{bmatrix}
        &&\partial_{\varphi} = \renewcommand\arraystretch{1.5}\begin{bmatrix}
                -\oph_{1}  & 0         & 0         & 0        & \cdots \\
                \od_{2}    & -\oph_{3} & 0         & 0        & \cdots \\
                0          & \od_{4}   & -\oph_{5} & 0        & \cdots \\
                \vdots     & 0         & \od_{6}   & -\oph_{7}& \ddots \\
                \vdots     & \vdots    & \vdots    & \ddots   & \ddots
        \end{bmatrix}
    \end{aligned}
    \end{equation}
    the block rows of $\partial_{d}$ being indexed by the odd $a$ and its block columns by the even $a$, and the other way round for $\partial_{\varphi}$. For $j\le n-1$ only the summands with $a\le j$ are present, so
    \[\partial_{j}=\begin{cases}
        \partial_{d}[(k+1)\times (k+1)] & j=2k+1 \text{ odd}, \\
        \partial_{\varphi}[(k+1)\times k] & j=2k \text{ even},
    \end{cases}\]
    where $\partial[p\times q]$ denotes the submatrix on the first $p$ block rows and the first $q$ block columns.

    \begin{remark}\label{Rem:squarezero}
        In Sections~\ref{S:whenstar} and~\ref{S:higher} we will use the following description of the matrices \eqref{Eq:matrices}, which are defined for any null-homotopy. Since the summand $\opp_{j-2k}$ of $M_{j}$ lies in row $k$, the underlying graded module of $\bul{M}$ is $\bigoplus_{k\ge0}\Sigma^{2k}\Pol$, with row $k$ equal to $\Sigma^{2k}\Pol$. Let $\xi$ be the endomorphism of degree $-2$ which maps row $k\ge1$ to row $k-1$ by the identity of $\Pol$ and vanishes on row $0$, and let endomorphisms of $\Pol$, such as $\od$ and $\oph$, act on each row as they act on $\Pol$. Then the matrices \eqref{Eq:matrices} are the components of
        \begin{equation}\label{Eq:rowform}
            \partial=\od-\oph\xi.
        \end{equation}
        Since $\xi$ commutes with $\od$ and with $\oph$, and $\od\oph+\oph\od=0$ by \eqref{Eq:anticomm}, it follows that
        \[\partial^{2}=\od^{2}-\left(\od\oph+\oph\od\right)\xi+\oph^{2}\xi^{2}=\oph^{2}\xi^{2}.\]
        Then since each summand $\opp_{a}$ occurs in row $2$ of $M_{a+4}$, the matrices \eqref{Eq:matrices} satisfy $\partial^{2}=0$ if and only if Hypothesis~\ref{Hyp:star} holds.
    \end{remark}

    \section{The hypothesis and matrix factorizations}\label{S:whenstar}
    In this section we study Hypothesis~\ref{Hyp:star} in more detail. By Remark~\ref{Rem:squarezero} the hypothesis holds if and only if the matrices \eqref{Eq:matrices} define a complex. We first show that it holds for any null-homotopy when $\bul{P}$ has length $n\le2$, and that it reduces to a single equation when $n=3$. Then we describe the hypothesis in terms of matrix factorizations of $x$ over $R$.

    Throughout this section $\bul{P}=(P_{i},d_{i})_{0\le i\le n}$ is a finite type $R$-projective resolution of $N\in\FPinf(S)$ of length $n<\infty$, and $\{\varphi_{i}\}_{i\ge1}$ is a null-homotopy for $\lambda_{x}:\bul{P}\to\bul{P}$ as in Lemma~\ref{Lem:shamash}. We write $P=\bigoplus_{i}P_{i}$ for the underlying graded module of $\bul{P}$, which is a finitely generated projective $R$-module with $P_{i}=0$ for $i<0$ and for $i>n$. An endomorphism of $P$ is \emph{homogeneous of degree $t$} if it maps $P_{i}$ to $P_{i+t}$ for all $i$, and every $R$-linear endomorphism of $P$ is a finite sum of homogeneous ones. Then the collections $\{d_{i}\}_{i\ge1}$ and $\{\varphi_{i}\}_{i\ge1}$ define homogeneous endomorphisms $d$ and $\varphi$ of $P$ of degrees $-1$ and $+1$ respectively, and the homotopy equation of Lemma~\ref{Lem:shamash} is equivalent to
    \begin{equation}\label{Eq:gradedhomotopy}
        d\varphi+\varphi d=x\cdot\id_{P}.
    \end{equation}
    We also grade $P$ over $\Z/2$ by the parity of $i$, with
    \[P^{\mathrm{ev}}=\bigoplus_{i\equiv 0\!\!\pmod{2}}P_{i},\qquad P^{\mathrm{odd}}=\bigoplus_{i\equiv 1\!\!\pmod{2}}P_{i},\]
    and we write $\opp^{\mathrm{ev}}=S\otimes_{R}P^{\mathrm{ev}}$ and $\opp^{\mathrm{odd}}=S\otimes_{R}P^{\mathrm{odd}}$ for their reductions modulo $x$. We say that an endomorphism of $P$ is \emph{even} (resp. \emph{odd}) if it preserves (resp. reverses) this grading, and both $d$ and $\varphi$ are odd. We write $\varphi^{2}=0$ if $\varphi_{i}\varphi_{i+1}=0$ for all $i\ge1$, and note that this is a stronger condition than Hypothesis~\ref{Hyp:star}, which only requires these compositions to vanish modulo $x$.

    \begin{proposition}\label{Prop:starlow}
        Let $\{\varphi_{i}\}_{i\ge1}$ be a null-homotopy for the chain map $\lambda_{x}:\bul{P}\to\bul{P}$.
        \begin{enumerate}[label=\textup{(\arabic*)}]
            \item If $n\le2$, then $\varphi_{i}\varphi_{i+1}=0$ for all $i\ge1$.
            \item If $n=3$, then there is a unique $f\in\Hom_{R}(P_{0},P_{3})$ with $\varphi_{1}\varphi_{2}=fd_{3}$, and moreover $\varphi_{2}\varphi_{3}=d_{1}f$. In particular $\oph_{1}\oph_{2}=0$ if and only if $f\in x\cdot\Hom_{R}(P_{0},P_{3})$, and in this case $\oph_{2}\oph_{3}=0$ as well, so that Hypothesis~\ref{Hyp:star} reduces to the single condition $\oph_{1}\oph_{2}=0$.
        \end{enumerate}
    \end{proposition}
    \begin{proof}
        Firstly, since $x\in R$ is central the endomorphism $\varphi$ commutes with multiplication by $x$, and by \eqref{Eq:gradedhomotopy} we have $\varphi d=x\cdot\id_{P}-d\varphi$ and $d\varphi=x\cdot\id_{P}-\varphi d$. Then
        \[\varphi^{2}d=\varphi\left(x\cdot\id_{P}-d\varphi\right)=x\varphi-\varphi d\varphi=\left(x\cdot\id_{P}-\varphi d\right)\varphi=d\varphi^{2},\]
        and by comparing the components of both sides on $P_{i-1}$ we obtain
        \begin{equation}\label{Eq:phi12d2}
            \left(\varphi_{i}\varphi_{i+1}\right)d_{i+1}=d_{i-1}\left(\varphi_{i-1}\varphi_{i}\right)\qquad\text{for all }i\ge1.
        \end{equation}
        Since $\varphi_{0}=0$, the case $i=1$ of \eqref{Eq:phi12d2} is simply $\varphi_{1}\varphi_{2}d_{2}=0$, and since $\bul{P}$ is exact at $P_{2}$ we have $\im(\varphi_{1}\varphi_{2})\subseteq\ker(d_{2})=\im(d_{3})$. Now if $n\le2$ then $P_{3}=0$, and hence $\varphi_{1}\varphi_{2}=0$. Furthermore, since $P_{i+1}=0$ for all $i\ge2$, we have $\varphi_{i}\varphi_{i+1}=0$ for all $i\ge2$, which proves (1).

        Now suppose that $n=3$. Then $P_{4}=0$, and since $\bul{P}$ is exact at $P_{3}$ the map $d_{3}$ is injective. Since $\im(\varphi_{1}\varphi_{2})\subseteq\im(d_{3})$ it follows that there is a unique $f\in\Hom_{R}(P_{0},P_{3})$ with $\varphi_{1}\varphi_{2}=fd_{3}$. Then by \eqref{Eq:phi12d2} with $i=2$ we have $\varphi_{2}\varphi_{3}d_{3}=d_{1}fd_{3}$, and since $d_{3}$ is injective this gives $\varphi_{2}\varphi_{3}=d_{1}f$.

        Finally, modulo $x$ these identities become $\oph_{1}\oph_{2}=\overline{f}\,\od_{3}$ and $\oph_{2}\oph_{3}=\od_{1}\overline{f}$. Since $\opp_{4}=0$ and $\Ho_{3}(\Pol)=0$ by Lemma~\ref{Lem:tori}, the map $\od_{3}$ is injective, and hence $\oph_{1}\oph_{2}=0$ if and only if $\overline{f}=0$, that is, if and only if $\im(f)\subseteq xP_{3}$. Then by Remark~\ref{Rem:xinj} this holds if and only if $f\in x\cdot\Hom_{R}(P_{0},P_{3})$, and in this case $\oph_{2}\oph_{3}=\od_{1}\overline{f}=0$ as well. Since $\varphi_{i}\varphi_{i+1}=0$ for all $i\ge3$, it follows that Hypothesis~\ref{Hyp:star} reduces to the single condition $\oph_{1}\oph_{2}=0$, as required.
    \end{proof}

    For larger $n$ we describe the hypothesis in terms of matrix factorizations. Following \cite[Def.~2.1]{dyckerhoff}, where the modules are free of finite rank over a regular local ring, we define a \emph{matrix factorization} of $x$ over $R$ to be a $\Z/2$-graded finitely generated projective $R$-module $X=X^{0}\oplus X^{1}$, together with an odd $R$-linear endomorphism $D$ of $X$ such that $D^{2}=x\cdot\id_{X}$. Equivalently, the components $X^{0}\to X^{1}$ and $X^{1}\to X^{0}$ of $D$ compose in either order to multiplication by $x$, which is the original formulation of Eisenbud \cite[\S5, p.49]{matfac}. If $X^{0}$ and $X^{1}$ are free of the same finite rank $r$, then $D$ is given by a pair of $r\times r$ matrices over $R$, and we say that the matrix factorization has \emph{rank} $r$.

    \begin{proposition}\label{Prop:mf}
        With the notation above, we have $\left(d+\varphi\right)^{2}=x\cdot\id_{P}+\varphi^{2}$. In particular:
        \begin{enumerate}[label=\textup{(\arabic*)}]
            \item If $\varphi^{2}=0$, then $\left(P^{\mathrm{ev}}\oplus P^{\mathrm{odd}},\,d+\varphi\right)$ is a matrix factorization of $x$ over $R$.
            \item Hypothesis~\ref{Hyp:star} holds if and only if $\left(d+\varphi\right)^{2}=x\cdot v$ for some even automorphism $v$ of $P$.
        \end{enumerate}
    \end{proposition}
    \begin{proof}
        Since $d^{2}=0$, it follows from \eqref{Eq:gradedhomotopy} that
        \[\left(d+\varphi\right)^{2}=d^{2}+\left(d\varphi+\varphi d\right)+\varphi^{2}=x\cdot\id_{P}+\varphi^{2}.\]
        Now if $\varphi^{2}=0$ then $\left(d+\varphi\right)^{2}=x\cdot\id_{P}$, and since $P$ is finitely generated projective and $d+\varphi$ is odd this proves (1).

        For (2), suppose first that $\left(d+\varphi\right)^{2}=x\cdot v$ for some endomorphism $v$ of $P$. Then $\varphi^{2}=x\cdot(v-\id_{P})$, and hence $\im(\varphi^{2})\subseteq xP$, which is Hypothesis~\ref{Hyp:star}. Conversely, suppose that Hypothesis~\ref{Hyp:star} holds, so that $\im(\varphi^{2})\subseteq xP$. Since multiplication by $x$ gives an isomorphism $P\cong xP$ which preserves the grading (Remark~\ref{Rem:xinj}), there is a unique endomorphism $\theta$ of $P$, homogeneous of degree $2$, with $\varphi^{2}=x\cdot\theta$. Then $\left(d+\varphi\right)^{2}=x\cdot v$ with $v=\id_{P}+\theta$, and $v$ is even since $\id_{P}$ and $\theta$ have even degree. Furthermore, $\theta^{k}$ maps $P_{i}$ to $P_{i+2k}$, which is zero when $2k>n$, and hence $\theta$ is nilpotent. It follows that $v$ is invertible with inverse $\sum_{k\ge0}(-\theta)^{k}$, as required.
    \end{proof}
    When $\varphi^{2}=0$, the matrix factorization of Proposition~\ref{Prop:mf}(1) determines the periodic part of the resolution $\bul{M}$ of Theorem~\ref{Thm:main}, as follows.
    \begin{corollary}\label{Cor:mftail}
        Suppose that $\varphi^{2}=0$. Then for all $j\ge n-1$ we have $M_{j}=\opp^{\mathrm{ev}}$ or $M_{j}=\opp^{\mathrm{odd}}$ according to the parity of $j$, and $\partial_{j}$ is the restriction of $\od-\oph$ to $M_{j}$ for all $j\ge n$. Moreover, the periodic part of $\bul{M}$ is isomorphic to the reduction modulo $x$ of the matrix factorization $\left(P^{\mathrm{ev}}\oplus P^{\mathrm{odd}},\,d+\varphi\right)$ of Proposition~\ref{Prop:mf}(1).
    \end{corollary}
    \begin{proof}
        Let $j\ge n-1$. By Lemma~\ref{Lem:closedform} the module $M_{j}$ is the direct sum of the modules $\opp_{a}$ with $0\le a\le n$ and $a\equiv j\pmod{2}$, where $\opp_{a}$ lies in row $(j-a)/2$. This gives the first assertion.

        Now let $j\ge n$ and let $\opp_{a}$ be a summand of $M_{j}$. If $\opp_{a+1}\neq0$ then $a+1\le n\le j$, and since $a+1$ and $j$ have opposite parity it follows that $a+1\le j-1$. Then $\opp_{a+1}$ lies in row $(j-a)/2-1\ge0$ of $M_{j-1}$, and hence by \eqref{Eq:rowform} the differential $\partial_{j}=\od-\oph\xi$ is the restriction of $\od-\oph$ to $M_{j}$.

        For the last assertion, since $\od$ preserves the row and $\xi$ lowers it by one, multiplying the summand in row $k$ by $(-1)^{k}$ for all $k\ge0$ gives an isomorphism from $\bul{M}$ to the complex with the same modules and the differential $\od+\oph\xi$. As in the previous paragraph, for $j\ge n$ the map $\od+\oph\xi$ on $M_{j}$ is the restriction of $\od+\oph$, which is the reduction modulo $x$ of $d+\varphi$, and this gives the required isomorphism.
    \end{proof}

    \section{Higher homotopies}\label{S:higher}
    In this section we show that eventual periodicity does not depend on Hypothesis~\ref{Hyp:star}. When the hypothesis fails, the null-homotopy $\{\varphi_{i}\}_{i\ge1}$ can be extended to a system of higher homotopies, and these give a resolution of $N$ with the same modules as $\bul{M}$. This construction is due to Shamash for minimal free resolutions over a commutative local ring \cite[\S3]{shamash1}, and it was extended to regular sequences by Eisenbud \cite[\S7]{matfac} (see also \cite[\S3.4 and Constr.~4.1.3]{ebudpeeva}). Since we apply it in Section~\ref{S:applications} to integral group rings, we record it here for projective resolutions over a ring $R$ which need not be commutative.

    Throughout this section $\bul{P}\to N$ is a finite type $R$-projective resolution of a module $N\in\SMod$, of length $n<\infty$, and we use the graded notation of Section~\ref{S:whenstar}.
    \begin{definition}[Shamash]\label{Def:higher}
        A \emph{system of higher homotopies} for $x$ on $\bul{P}$ is a collection $\{h_{m}\}_{m\ge0}$ of $R$-linear endomorphisms of $P$, with $h_{m}$ homogeneous of degree $2m-1$ for all $m\ge0$, such that $h_{0}=d$ and
        \begin{equation}\label{Eq:higher}
            \sum_{k+l=m}h_{k}h_{l}=\begin{cases}x\cdot\id_{P} & m=1,\\ 0 & m\neq1,\end{cases}
        \end{equation}
        for all $m\ge0$.
    \end{definition}
    The relation of index $0$ is $d^{2}=0$. Put $\tau_{1}\coloneqq x\cdot\id_{P}$ and $\tau_{m}\coloneqq-\sum_{k=1}^{m-1}h_{k}h_{m-k}$ for $m\ge2$. Then for $m\ge1$ the relation of index $m$ reads
    \begin{equation}\label{Eq:higherrec}
        dh_{m}+h_{m}d=\tau_{m}.
    \end{equation}
    In particular $h_{1}$ is a null-homotopy for $\lambda_{x}$, and $dh_{2}+h_{2}d=-h_{1}^{2}$. The maps $h_{m}$ and $\tau_{m}$ are those of \cite[\S3]{shamash1}, whose relations are given in \cite[Rem.~1, p.458]{shamash1} (see also \cite[Def.~3.4.1]{ebudpeeva}). We write $\ohh{m}=S\otimes_{R}h_{m}$ for the reduction of $h_{m}$ modulo $x$, so that $\ohh{0}=\od$, and $\ohh{1}=\oph$ if $h_{1}=\varphi$.

    \begin{lemma}\label{Lem:dsq}
        Let $\{h_{m}\}_{m\ge0}$ be endomorphisms of $P$ with $h_{0}=d$ and $h_{m}$ homogeneous of degree $2m-1$ for all $m\ge0$, and put $D=\sum_{m\ge0}h_{m}$. Then $\{h_{m}\}_{m\ge0}$ is a system of higher homotopies for $x$ on $\bul{P}$ if and only if $D^{2}=x\cdot\id_{P}$, and in this case $\left(P^{\mathrm{ev}}\oplus P^{\mathrm{odd}},D\right)$ is a matrix factorization of $x$ over $R$.
    \end{lemma}
    \begin{proof}
        Firstly, because $h_{m}=0$ for $2m-1>n$, the sum $D$ is finite. Then since $h_{k}h_{l}$ is homogeneous of degree $2(k+l)-2$ we have
        \[D^{2}=\sum_{m\ge0}\,\sum_{k+l=m}h_{k}h_{l},\]
        in which the $m$-th term is the homogeneous component of $D^{2}$ of degree $2m-2$. Hence $D^{2}=x\cdot\id_{P}$ if and only if \eqref{Eq:higher} holds for all $m\ge0$. Finally, since each $h_{m}$ has odd degree, $D$ is odd, and since $P^{\mathrm{ev}}$ and $P^{\mathrm{odd}}$ are finitely generated projective, this gives the last assertion.
    \end{proof}
    In particular, Proposition~\ref{Prop:mf}(1) is the case of Lemma~\ref{Lem:dsq} in which $h_{1}=\varphi$ and $h_{m}=0$ for all $m\ge2$.

    To construct a system of higher homotopies we will solve the equations \eqref{Eq:higherrec} inductively, by means of the following lemma (cf. \cite[Lem.~2, p.458]{shamash1}).
    \begin{lemma}\label{Lem:lift}
        Let $f$ be an endomorphism of $P$ which is homogeneous of degree $t\ge1$ and satisfies $df=fd$. Then there is an endomorphism $g$ of $P$, homogeneous of degree $t+1$, with $dg+gd=f$.
    \end{lemma}
    \begin{proof}
        Write $f_{i}:P_{i}\to P_{i+t}$ for the components of $f$. We construct the components $g_{i}:P_{i}\to P_{i+t+1}$ of $g$ by induction on $i\ge0$, with $g_{i}=0$ for $i<0$, such that
        \[d_{i}g_{i-1}+g_{i}d_{i+t+1}=f_{i}\qquad(i\ge0).\]
        Suppose that $g_{0},\ldots,g_{i-1}$ have been constructed, and put $u_{i}=f_{i}-d_{i}g_{i-1}$. Then by the equation in degree $i-1$, together with $d_{i}d_{i-1}=0$ and $fd=df$, we have
        \[u_{i}d_{i+t}=f_{i}d_{i+t}-d_{i}\left(f_{i-1}-d_{i-1}g_{i-2}\right)=f_{i}d_{i+t}-d_{i}f_{i-1}=0.\]
        Now the complex $\bul{P}$ is exact at $P_{i+t}$ because $i+t\ge1$, and hence $\im(u_{i})\subseteq\im(d_{i+t+1})$. Then since $P_{i}$ is projective, there is a lift $g_{i}:P_{i}\to P_{i+t+1}$ with $g_{i}d_{i+t+1}=u_{i}$.
    \end{proof}

    The following theorem is the construction of Shamash in our setting. In contrast to \cite[\S3]{shamash1} and \cite[\S7]{matfac}, the maps $h_{m}$ are obtained from Lemma~\ref{Lem:lift} rather than from a basis of each module, and the centrality of $x\in R$ replaces commutativity in the proof that $d\tau_{m}=\tau_{m}d$. Up to the signs of the rows, the resolution of part (2) is the complex of \cite[Eq.~(3), p.460]{shamash1} (cf. \cite[Thm.~7.2]{matfac} and \cite[Constr.~4.1.3]{ebudpeeva}).
    \begin{theorem}[Shamash]\label{Thm:higher}
        Let $\{\varphi_{i}\}_{i\ge1}$ be a null-homotopy for the chain map $\lambda_{x}:\bul{P}\to \bul{P}$.
        \begin{enumerate}[label=\textup{(\arabic*)}]
            \item There is a system of higher homotopies $\{h_{m}\}_{m\ge0}$ for $x$ on $\bul{P}$ with $h_{1}=\varphi$.
            \item For any such system, the graded module $\bul{L}=\bigoplus_{k\ge0}\Sigma^{2k}\Pol$ with the differential
            \[\partial=\sum_{m\ge0}(-1)^{m}\,\ohh{m}\,\xi^{m},\]
            where $\xi$ is the row shift of Remark~\ref{Rem:squarezero}, is a finite type $S$-projective resolution of $N$. For all $j\ge n-1$ we have $L_{j}=\opp^{\mathrm{ev}}$ or $L_{j}=\opp^{\mathrm{odd}}$ according to the parity of $j$, and $\partial_{j}$ is the restriction of $\sum_{m\ge0}(-1)^{m}\ohh{m}$ to $L_{j}$ for all $j\ge n$. In particular $\bul{L}$ is $2$-periodic in degrees $\ge n-1$.
            \item The periodic part of $\bul{L}$ is isomorphic to the reduction modulo $x$ of the matrix factorization $\left(P^{\mathrm{ev}}\oplus P^{\mathrm{odd}},\,\sum_{m\ge0}h_{m}\right)$ of Lemma~\ref{Lem:dsq}.
        \end{enumerate}
        If $h_{m}=0$ for all $m\ge2$, then $\bul{L}$ is the resolution $\bul{M}$ of Theorem~\ref{Thm:main}.
    \end{theorem}
    \begin{proof}
        For (1), put $h_{0}=d$ and $h_{1}=\varphi$. Then \eqref{Eq:higherrec} holds for $m=1$ by \eqref{Eq:gradedhomotopy}. Now let $m\ge2$ and suppose that $h_{1},\ldots,h_{m-1}$ satisfy \eqref{Eq:higherrec} for all indices less than $m$. Since $\tau_{m}$ is homogeneous of degree $2m-2\ge2$, by Lemma~\ref{Lem:lift} it suffices to show that $d\tau_{m}=\tau_{m}d$. Applying \eqref{Eq:higherrec} first for the index $k$ and then for the index $m-k$, we have
        \begin{align*}
            d\tau_{m}=-\sum_{k=1}^{m-1}dh_{k}h_{m-k}&=\sum_{k=1}^{m-1}\left(h_{k}d-\tau_{k}\right)h_{m-k}\\
            &=\sum_{k=1}^{m-1}h_{k}\left(\tau_{m-k}-h_{m-k}d\right)-\sum_{k=1}^{m-1}\tau_{k}h_{m-k}\\
            &=\tau_{m}d+\sum_{k=1}^{m-1}\left(h_{k}\tau_{m-k}-\tau_{k}h_{m-k}\right).
        \end{align*}
        Moreover, since $x\in R$ is central we have $xh_{m-1}=h_{m-1}x$, and hence
        \[\sum_{k=1}^{m-1}\tau_{k}h_{m-k}=xh_{m-1}-\sum_{a+b+c=m}h_{a}h_{b}h_{c}=\sum_{k=1}^{m-1}h_{k}\tau_{m-k},\]
        where the sum over $a+b+c=m$ runs over $a,b,c\ge1$. It follows that $d\tau_{m}=\tau_{m}d$, as required.

        For (2), since each $\ohh{m}$ acts on every row as it acts on $\Pol$, the maps $\ohh{m}$ commute with $\xi$. Then since $x$ acts as zero on the reduced complex $\Pol$, it follows from \eqref{Eq:higher} that
        \[\partial^{2}=\sum_{p,q\ge0}(-1)^{p+q}\,\ohh{p}\,\ohh{q}\,\xi^{p+q}=\sum_{m\ge0}(-1)^{m}\Big(\sum_{p+q=m}\ohh{p}\,\ohh{q}\Big)\xi^{m}=0.\]
        Moreover, $\xi$ commutes with $\partial$, and hence it is a surjective chain map $\bul{L}\to\Sigma^{2}\bul{L}$ whose kernel is row $0$, on which $\partial=\od$. This gives a short exact sequence of complexes
        \[
            \begin{tikzcd}
                0 \arrow[r] & \Pol \arrow[r] & \bul{L} \arrow[r, "\xi"] & \Sigma^{2}\bul{L} \arrow[r] & 0
            \end{tikzcd}
        \]
        (cf. \cite[Prop.~4.3.2]{ebudpeeva}). Then since $\Ho_{j}(\Pol)=0$ for $j\ge2$ by Lemma~\ref{Lem:tori}, its long exact sequence consists of isomorphisms $\xi_{\ast}:\Ho_{j}(\bul{L})\to\Ho_{j-2}(\bul{L})$ for $j\ge3$, together with an exact sequence
        \[
            \begin{tikzcd}
                0 \arrow[r] & \Ho_{2}(\bul{L}) \arrow[r] & \Ho_{0}(\bul{L}) \arrow[r, "\delta"] & \Ho_{1}(\Pol) \arrow[r] & \Ho_{1}(\bul{L}) \arrow[r] & 0.
            \end{tikzcd}
        \]
        Since $L_{0}=\opp_{0}$ and $L_{1}=\opp_{1}$ lie in row $0$ and $\partial_{1}=\od_{1}$, we have $\Ho_{0}(\bul{L})=\Ho_{0}(\Pol)$. Moreover, row $1$ of $L_{2}$ is $\opp_{0}$, and the component of $\partial$ from this summand to row $0$ is $-\oph_{1}$. Hence $\delta=-\Phi_{\ast}$, which is an isomorphism by Lemma~\ref{Lem:shamash}, and it follows that $\Ho_{1}(\bul{L})=\Ho_{2}(\bul{L})=0$. Then the isomorphisms $\xi_{\ast}$ give $\Ho_{j}(\bul{L})=0$ for all $j\ge1$. Since each $L_{j}$ is a finitely generated projective $S$-module, it follows that $\bul{L}$ is a finite type $S$-projective resolution of $N$.

        For the rest of (2), let $j\ge n-1$. Then $L_{j}=\bigoplus_{k\ge0}\opp_{j-2k}$ is the direct sum of the modules $\opp_{a}$ with $0\le a\le n$ and $a\equiv j\pmod{2}$, where $\opp_{a}$ lies in row $(j-a)/2$. Now let $j\ge n$, and let $\opp_{a}$ be a summand of $L_{j}$. For $m\ge0$ the map $\ohh{m}$ sends $\opp_{a}$ to $\opp_{a+2m-1}$, and if $\opp_{a+2m-1}\neq0$ then $a+2m-1\le n\le j$. Since $a+2m-1$ and $j$ have opposite parity, it follows that $a+2m-1\le j-1$, and hence $\opp_{a+2m-1}$ lies in row $(j-a)/2-m\ge0$ of $L_{j-1}$. Then $\partial_{j}$ is the restriction of $\sum_{m\ge0}(-1)^{m}\ohh{m}$ to $L_{j}$, as required.

        For (3), since the component of $\partial$ from row $k$ to row $k-m$ is $(-1)^{m}\ohh{m}$, multiplying the summand in row $k$ by $(-1)^{k}$ for all $k\ge0$ gives an isomorphism from $\bul{L}$ to the complex with the same modules and the differential $\sum_{m\ge0}\ohh{m}\,\xi^{m}$. By the argument for (2), for $j\ge n$ this differential on $L_{j}$ is the restriction of $\sum_{m\ge0}\ohh{m}$, which is the reduction modulo $x$ of $\sum_{m\ge0}h_{m}$, and this gives the required isomorphism.

        Finally, suppose that $h_{m}=0$ for all $m\ge2$. Then the relation of index $2$ reads $\varphi^{2}=0$, and $\partial=\od-\oph\xi$ is the differential \eqref{Eq:rowform} of $\bul{M}$, so that $\bul{L}=\bul{M}$.
    \end{proof}

    \section{Application to integral group rings}\label{S:applications}
    We now specialize the constructions of Sections~\ref{S:gencons} and~\ref{S:higher} to $R=\ZG$, the integral group ring of an infinite discrete group. Since the construction requires a module of finite projective dimension over $\ZG$, we begin by recalling the relevant finiteness properties, following \cite[Ch.~VIII]{brown} and \cite[\S I]{bieri}.

    \subsection{Finiteness properties and free differential calculus}
    \begin{definition}\label{Def:FP}
		A group $G$ is of \emph{type $\FP$}, respectively of \emph{type $\FL$}, if the trivial module $\Z\in\ZGmod$ admits a finite type projective, respectively free, resolution of finite length. The \emph{cohomological dimension} of $G$ is $\cd(G)\coloneqq\pd_{\ZG}(\Z)$. If $G$ has a torsion-free subgroup of finite index, then every such subgroup has the same cohomological dimension, and this common value is the \emph{virtual cohomological dimension} $\vcd(G)$ of $G$ \cite[Prop.~5, p.86]{serredisc}.
	\end{definition}
    A group of type $\FP$ has finite cohomological dimension, and hence it is torsion-free \cite[Cor.~VIII.2.5, p.187]{brown}. For such a group, with $\cd(G)=n$ and $\bul{P}\to\Z$ a finite type projective resolution of length $n$, the \emph{Euler characteristic} is
    \[\chi(G)\coloneqq\sum_{i=0}^{n}(-1)^{i}\rank_{\ZG}(P_{i}),\qquad
      \rank_{\ZG}(P_{i})\coloneqq\rank_{\Z}(\Z\otimes_{\ZG}P_{i}),\]
    and it does not depend on the resolution chosen (see \cite[\S IX.7]{brown}).

    A group of type $\FL$ with $n=\cd(G)<\infty$ has a finite type free resolution of length $n$ (cf. \cite[\S VIII.6]{brown}). To see this, let $0\to F_{m}\to\cdots\to F_{0}\to\Z\to0$ be a finite type free resolution, so that $m\ge n$. Since $\cd(G)=n$, the $n$-th syzygy module $\Omega^{n}$ of this resolution is projective, and since the tail $0\to F_{m}\to\cdots\to F_{n}\to\Omega^{n}\to0$ is a finite free resolution of $\Omega^{n}$, it follows that $\Omega^{n}$ is stably free, so that $\Omega^{n}\oplus\ZG^{a}\cong\ZG^{b}$ for some $a,b\in\N$. Then if we truncate $\bul{F}$ in degree $n$ and add the trivial complex $0\to\ZG^{a}\xrightarrow{\ \id\ }\ZG^{a}\to0$ in degrees $n$ and $n-1$, we obtain the desired free resolution.

    Next we bound the projective dimension of an arbitrary $\ZG$-module, which makes the hypothesis $\pd_{R}(N)<\infty$ of Section~\ref{S:gencons} available over a group ring.
     \begin{lemma}\label{Lem:gdim}
    	Let $G$ be a group with $\cd(G)<\infty$. Then $\glob(\ZG)\le \cd(G)+1$. The same bound holds for right modules, since $g\mapsto g^{-1}$ gives a ring isomorphism $\ZG\cong\ZG^{\mathrm{op}}$.
    \end{lemma}
    \begin{proof}
    	For any $A\in\ZGmod$ there is a short exact sequence of $\ZG$-modules
    	\[
    	\begin{tikzcd}
    		0 \arrow[r] & K \arrow[r] & \ZG^{(A)} \arrow[r] & A \arrow[r] & 0,
    	\end{tikzcd}
    	\]
    	where $\ZG^{(A)}$ is the free $\ZG$-module on the underlying set of $A$. Then since $\ZG\in\Zmod$ is free and $\Z$ is a hereditary ring, both $\ZG^{(A)}$ and its subgroup $K$ are free as abelian groups. By dimension shifting (see \cite[\S III.7, p.74]{brown}) we have $\Ext^{i+1}_{\ZG}(A,B)\cong\Ext^{i}_{\ZG}(K,B)$ for all $i\ge1$ and all $B\in\ZGmod$, and hence
    	\[\pd_{\ZG}(A)\le \pd_{\ZG}(K)+1.\]
    	It suffices to show that $\pd_{\ZG}(K)\le\cd(G)$ for any $K\in\ZGmod$ which is free as an abelian group. To this end, let $\bul{Q}\to\Z$ be a $\ZG$-projective resolution of length $\cd(G)<\infty$. Then since $K$ is free abelian, it is flat as a $\Z$-module and the functor $-\otimes_{\Z}K$ is exact. By \cite[Cor.~III.5.7, p.69]{brown} each of the modules $Q_{i}\otimes_{\Z}K$ is projective (with the diagonal $G$-action), and hence $\bul{Q}\otimes_{\Z}K\to K$ is a projective resolution of length $\cd(G)$, which gives $\pd_{\ZG}(K)\le\cd(G)$, as required.
    \end{proof}

    In Section~\ref{S:examples} the low-degree terms of each resolution are given by Lyndon's presentation of the relation module \cite{lyndon} in terms of Fox's free differential calculus \cite{fox}.
    \begin{theorem}[Fox 1953]\label{Thm:foxdif}
    	Let $F$ be the free group on $\{y_{1},\ldots,y_{k}\}$. For each $i\in\{1,\ldots,k\}$ there is a unique \emph{derivation} $\foxi:F\to\ZF$, that is, a map satisfying $\Foxi{(uv)}=\Foxi{u}+u\cdot\Foxi{v}$ for all $u,v\in F$, with $\Foxi{y_{j}}=\delta_{ij}$. These \emph{Fox derivatives} satisfy the \emph{fundamental formula}
    	\[w-1 = \sum_{i=1}^{k}\Foxi{w}(y_{i}-1)\qquad (w\in F),\]
    	and for any quotient $\om:F\to G$ the fundamental formula remains valid in $\ZG$ after applying the induced map $\om^{*}:\ZF\to\ZG$.
    \end{theorem}

    \begin{theorem}[Lyndon 1950]\label{Thm:preres}
    	Suppose $G=\langle y_{1},\ldots, y_{p} \mid r_{1},\ldots, r_{q}\rangle$ is a finitely presented group. Then the trivial module $\Z\in\ZGmod$ has a free resolution beginning with
        \begin{equation*}
            \begin{tikzcd}
        		\ZG^{q} \arrow[r,"d_{2}"] & \ZG^{p} \arrow[r, "d_{1}"] & \ZG \arrow[r, "\varepsilon"] & \Z \arrow[r] & 0,
        	\end{tikzcd}
        \end{equation*}
        where $d_{2}:\ZG^{q}\to\ZG^{p}$ and $d_{1}:\ZG^{p}\to \ZG$ are given by
        \[d_{1}=\begin{bmatrix}
            y_{1}-1 \\
            \vdots \\
            y_{p}-1
        \end{bmatrix},\qquad d_{2} = \left[\;\om^{*}\!\left(\Fox{r_{j}}{y_{i}}\right)\right]_{\substack{1\le j\le q\\ 1\le i\le p}},\]
        with $\om^{*}:\ZF\to\ZG$ induced by the quotient $\om:F\to G$ from the free group $F$ on $y_{1},\ldots,y_{p}$. The image of $d_{2}$ is the \emph{relation module} $\mathfrak{R}/\mathfrak{R}\pr$, where $\mathfrak{R}=\ker(\om)$ and $\mathfrak{R}\pr=[\mathfrak{R},\mathfrak{R}]$, and it is generated by the classes of the relators $r_{1},\ldots,r_{q}$.
    \end{theorem}

    To see that the free modules in the periodic part of the resolution of Theorem~\ref{Thm:zg}(2) all have the same rank, we will use the following classical result. Gottlieb proved that a finite aspherical complex whose fundamental group has non-trivial center has vanishing Euler characteristic \cite{gottlieb}, and Stallings gave the following algebraic formulation \cite[Thm.~2]{stallingsgott}. Rosset later replaced the center by any non-trivial normal abelian subgroup \cite{rosset}, but we will only need the case of the center.
    \begin{theorem}[Gottlieb, Stallings]\label{Thm:gottlieb}
        Let $G$ be a group of type $\FL$ with non-trivial center. Then $\chi(G)=0$.
    \end{theorem}

    For $g$ in a group $G$ and $k\in\N$ we write
    \[\Ngk{g}{k}\coloneqq\sum_{i=0}^{k}g^{i}\in\ZG.\]
    If $g$ is one of the generators of a presentation, then $\Ngk{g}{k}$ is the Fox derivative of the relator $g^{k+1}$ with respect to $g$, and these elements recur in all of the calculations made later.

    For a finitely presented group with a central element $c$, the component in degree $1$ of a null-homotopy for multiplication by $1-c$ is given directly by the Fox derivatives of a word which represents $c$.
    \begin{remark}\label{Rem:foxhomotopy}
        Let $G=\langle y_{1},\ldots, y_{p}\mid r_{1},\ldots, r_{q}\rangle$ be finitely presented, let $\bul{F}\to\Z$ be the free resolution of Theorem~\ref{Thm:preres}, let $1\neq c\in Z(G)$ be represented by a word $w$ in $y_{1},\ldots,y_{p}$, and put $x=1-c$. Then the map
        \[\varphi_{1}=-\begin{bmatrix}\om^{*}\!\left(\Fox{w}{y_{1}}\right) & \cdots & \om^{*}\!\left(\Fox{w}{y_{p}}\right)\end{bmatrix}:F_{0}=\ZG\longrightarrow F_{1}=\ZG^{p}\]
        satisfies $\varphi_{1}d_{1}=x\cdot\id_{F_{0}}$, which is the homotopy equation in degree $0$. Indeed, by the fundamental formula of Theorem~\ref{Thm:foxdif}, we have
        \[\varphi_{1}d_{1}=-\sum_{i=1}^{p}\om^{*}\!\left(\Fox{w}{y_{i}}\right)\left(y_{i}-1\right)=-(c-1)=x.\]
    \end{remark}

    \subsection{Eventually \texorpdfstring{$2$}{2}-periodic resolutions}
    Let $G$ be a group with $\cd(G)<\infty$, let $c\in Z(G)$ be a non-trivial central element and put $\G\coloneqq G/\langle c\rangle$, so that there is a central extension
    \begin{equation}\label{ses:VFP}
        \begin{tikzcd}
        	1 \arrow[r] & \langle c \rangle \arrow[r] & G \arrow[r] & \G \arrow[r] & 1.
        \end{tikzcd}
    \end{equation}
    Moreover, since $\cd(G)<\infty$, the group $G$ is torsion-free \cite[Cor.~VIII.2.5, p.187]{brown}, and hence $\langle c\rangle$ is infinite cyclic.

    To see how this fits with Assumption~\ref{ass:ring}, take $R=\ZG$ and set $x\coloneqq1-c$. Since $c$ is central in $G$, the element $x\in\ZG$ is central. By \cite[\S3.3, pp.134--138]{miles} there is a ring isomorphism $\ZG/(x)\cong\Z\G\neq0$, so $x$ is not a unit. Furthermore, since $\ZG$ is free as a module over the integral domain $\Z[c,c^{-1}]$ on any set of coset representatives of $\langle c\rangle$ in $G$, the element $x$ is a non-zero-divisor. Then with $S=\Z\G$ we are in the setting of Sections~\ref{S:gencons} and~\ref{S:higher}, and every $\ZG$-module has finite projective dimension by Lemma~\ref{Lem:gdim}.

    \begin{theorem}\label{Thm:zg}
    	Let $\G=G/\langle c\rangle$ be as in \eqref{ses:VFP}.
        \begin{enumerate}[label=\textup{(\arabic*)}]
            \item Let $0\neq N\in\FPinf(\Z\G)$ and put $n=\pd_{\ZG}(N)$. Then $N$ has a finite type $\Z\G$-projective resolution $\bul{M}=(M_{j},\partial_{j})_{j\ge 0}$ which is $2$-periodic in degrees $\ge n-1$.

            \item Suppose in addition that $G$ is of type $\FL$ with $n=\cd(G)$, and let $\bul{F}\to\Z$ be a finite type free $\ZG$-resolution of length $n$. Then the construction of $(1)$ applied to $\bul{F}$ gives a free resolution $\bul{M}$ of $\Z\in\ZGamod$, and the integer
            \[r\coloneqq\sum\limits_{i\equiv 0\!\!\pmod{2}}\rank_{\ZG}(F_{i}) = \sum\limits_{i\equiv 1\!\!\pmod{2}}\rank_{\ZG}(F_{i})\]
            satisfies $M_{j}\cong \Z\G^{r}$ for all $j\ge n-1$. We call $r$ the \emph{stable rank} of $\bul{M}$.
        \end{enumerate}
    \end{theorem}
    \begin{proof}
    	For (1), Lemma~\ref{Lem:gdim} gives $n=\pd_{\ZG}(N)<\infty$, so $N$ has a finite type $\ZG$-projective resolution of length $n$ by Corollary~\ref{Cor:lengthn}. The resolution $\bul{M}$ is then given by Theorems~\ref{Thm:imc} and~\ref{Thm:main} under Hypothesis~\ref{Hyp:star}, and by Theorem~\ref{Thm:higher} otherwise.

        For (2), we apply the iterated mapping cone construction of Section~\ref{S:gencons} to the resolution $\bul{F}$ itself, or the construction of Section~\ref{S:higher} when Hypothesis~\ref{Hyp:star} fails. Write $\Fol=\bul{F}/x\bul{F}$. In either case the modules in the periodic part are $M_{j}=\bigoplus_{a\equiv j\!\!\pmod{2}}\off_{a}$ for all $j\ge n-1$, and the two constructions differ only in their differentials. Since every $F_{i}$ is free, so is every $M_{j}$. Then the ranks of $M_{n-1}$ and $M_{n}$ are $\sum_{i\equiv 0\!\!\pmod{2}}\rank_{\ZG}(F_{i})$ and $\sum_{i\equiv 1\!\!\pmod{2}}\rank_{\ZG}(F_{i})$, in one order or the other according to the parity of $n$.
        
        Now since $c\in G$ is a non-trivial central element, it follows from Theorem~\ref{Thm:gottlieb} that $\chi(G)=0$, and by computing $\chi(G)$ from $\bul{F}$ we obtain
    	\[
            \sum_{i=0}^{n}(-1)^{i}\rank_{\ZG}(F_{i}) = \chi(G)=0 \iff  \sum_{i\equiv 0\!\!\pmod{2}}\rank_{\ZG}(F_{i})= \sum_{i\equiv 1\!\!\pmod{2}}\rank_{\ZG}(F_{i}).\qedhere
        \]
    \end{proof}
    If we add the complex $0\to\ZG^{a}\xrightarrow{\ \id\ }\ZG^{a}\to0$ to $\bul{F}$ in degrees $n$ and $n-1$, then both sums increase by $a$. Hence the stable rank $r$ depends on the choice of $\bul{F}$, and it is not an invariant of $\G$.

    \begin{corollary}\label{Cor:mfzg}
        Let $G$, $c$ and $\bul{F}$ be as in Theorem~\ref{Thm:zg}(2), and let $\{h_{m}\}_{m\ge0}$ be a system of higher homotopies for $x=1-c$ on $\bul{F}$. Then $\left(F^{\mathrm{ev}}\oplus F^{\mathrm{odd}},\,\sum_{m\ge0}h_{m}\right)$ is a matrix factorization of $1-c$ over $\ZG$ of rank $r$, and its reduction modulo $x$ is isomorphic to the periodic part of the resolution $\bul{L}$ of Theorem~\ref{Thm:higher}. If $\lambda_{x}$ has a null-homotopy $\varphi$ with $\varphi^{2}=0$, then we may take $h_{1}=\varphi$ and $h_{m}=0$ for all $m\ge2$, and in this case $\bul{L}$ is the resolution $\bul{M}$ of Theorem~\ref{Thm:zg}(2).
    \end{corollary}
    \begin{proof}
        Put $D=\sum_{m\ge0}h_{m}$. By Lemma~\ref{Lem:dsq} the pair $\left(F^{\mathrm{ev}}\oplus F^{\mathrm{odd}},\,D\right)$ is a matrix factorization of $x$, and by Theorem~\ref{Thm:higher}(3) its reduction modulo $x$ is isomorphic to the periodic part of $\bul{L}$. Since $F^{\mathrm{ev}}$ and $F^{\mathrm{odd}}$ are free of rank $r$ by Theorem~\ref{Thm:zg}(2), the map $D$ is given by a pair of $r\times r$ matrices over $\ZG$. Finally, if $\varphi^{2}=0$ then $h_{0}=d$, $h_{1}=\varphi$ and $h_{m}=0$ for all $m\ge2$ form a system of higher homotopies by Proposition~\ref{Prop:mf}(1) and Lemma~\ref{Lem:dsq}, and $\bul{L}=\bul{M}$ by the last part of Theorem~\ref{Thm:higher}.
    \end{proof}
    \begin{remark}\label{Rem:gysin}
        Consider the Lyndon--Hochschild--Serre spectral sequence of \eqref{ses:VFP} (see \cite[\S VII.6]{brown}). Since $\langle c\rangle$ is infinite cyclic, the $E_{2}$-page of this spectral sequence has only the two rows $q=0$ and $q=1$, and since $c\in G$ is central we have $\Ho^{1}(\langle c\rangle,A)\cong A$ as $\Z\G$-modules for any $\Z\G$-module $A$. The resulting long exact sequence shows that cup product with the class $e\in\Ho^{2}(\G,\Z)$ of the extension is an isomorphism $\Ho^{j}(\G,A)\to\Ho^{j+2}(\G,A)$ for all $\Z\G$-modules $A$ and all $j\ge\cd(G)$ (cf. \cite[Ex.~6.8.6, p.197]{weibel}). In this way the eventual periodicity of the cohomology of $\G$ can be seen without a resolution, and Theorem~\ref{Thm:zg} gives an explicit resolution in which this periodicity is visible.

        There is also a geometric picture when $\bul{F}$ is the cellular chain complex of the universal cover $X$ of a finite $K(G,1)$. The group $\langle c\rangle$ acts freely on the contractible complex $X$, so $X/\langle c\rangle$ is homotopy equivalent to the circle, as observed by Mislin and Talelli \cite[proof of Cor.~5.3]{mislintalelli}. The quotient $\G$ acts freely and cocompactly on $X/\langle c\rangle$, whose cellular chain complex is $\Fol=\Z\G\otimes_{\ZG}\bul{F}$. In this case Lemma~\ref{Lem:tori} computes the homology of the circle, and the stable rank $r$ counts the $\G$-orbits of even-dimensional cells of $X/\langle c\rangle$.
    \end{remark}
    
    \section{Explicit calculations}\label{S:examples}
    In this section we apply Theorem~\ref{Thm:zg} to four families of groups. In each of the explicit calculations below $G$ has a finite free resolution $\bul{F}$ of length $n\le3$, and the component in degree $1$ of the null-homotopy is given by the Fox derivatives of Remark~\ref{Rem:foxhomotopy}. In each of these calculations the null-homotopy also satisfies $\varphi^{2}=0$ over $\ZG$, so that Hypothesis~\ref{Hyp:star} holds and the differentials of the resolution are given by the matrices of Theorem~\ref{Thm:main}. Moreover, by Corollary~\ref{Cor:mfzg} the higher homotopies $h_{0}=d$ and $h_{1}=\varphi$ give a matrix factorization of $x$ over $\ZG$ whose reduction modulo $x$ is isomorphic to the periodic part of the resolution.

    Each resolution is then used to compute the integral homology and cohomology of the quotient group. Applying the augmentation map $\varepsilon:\Z\G\to\Z$ to the entries of the differentials, we obtain a complex of free abelian groups which computes $\Ho_{\ast}(\G,\Z)$, and the transposed matrices compute $\Ho^{\ast}(\G,\Z)$. Both are given by the Smith normal forms of these integer matrices, for which see \cite[\S2.4.4]{smithnorm}. We write $\varepsilon(\partial_{j})$ for the integer matrix obtained by applying $\varepsilon$ to the entries of $\partial_{j}$.

    We first consider the simplest case. Let $G=\langle t\rangle$ be infinite cyclic, so that $R=\ZG=\Z[t,t^{-1}]$ and $\Z\in\RMod$ has the free resolution $0\to R\xrightarrow{t-1}R\to\Z\to0$ of length $n=1$. For $m\ge2$ the element $c=t^{m}$ is central and non-trivial, the quotient is $\G=G/\langle c\rangle\cong\Z/m$, and $S=\Z\G$. Since $x=1-t^{m}$ factors as $-(t-1)\Ngk{t}{m-1}$, the map $\varphi_{1}=-\Ngk{t}{m-1}$ is a null-homotopy for $\lambda_{x}$. Then Hypothesis~\ref{Hyp:star} holds by Proposition~\ref{Prop:starlow}(1), and Theorem~\ref{Thm:main} gives the familiar $2$-periodic free resolution of $\Z$ over $\Z[\Z/m]$, whose differentials alternate between $t-1$ and $\Ngk{t}{m-1}$ as in \cite[I.6.3, p.21]{brown}.

    \subsection{Amalgamated products of cyclic groups}\label{ss:amalgam}
    Suppose throughout this subsection that $p,q\ge2$ and $\gcd(p,q)=1$. The map $S^{1}\to S^{1}\times S^{1}$ given by $z\mapsto(z^{p},z^{q})$ is an embedding, and its composition with a standard embedding of the torus in $\R^{3}$ is the $(p,q)$\emph{-torus knot} $K_{p,q}$, which wraps the torus $p$ times longitudinally and $q$ times meridionally \cite[\S3.C]{rolf}.
    \begin{definition}
        The ($p,q$)\emph{-torus knot group} is the one-relator group
        \[G_{p,q}=\langle a,b\mid a^{p}=b^{q}\rangle\cong \pi_{1}(\R^{3}\setminus K_{p,q}).\]
    \end{definition}
    By \cite[Prop.~3.38, p.49]{burde} the center of $G_{p,q}$ is infinite cyclic generated by $a^{p}=b^{q}$, and hence Theorem~\ref{Thm:zg} gives resolutions over the central quotient $G_{p,q}/\langle a^{p}\rangle\cong \langle a,b\mid a^{p}=b^{q}=1\rangle\cong \Z/p * \Z/q$. The same construction applies to the quotient by any positive power of the central generator. For $k\in\N$ with $k\ge1$ put $\G_{k}\coloneqq G_{p,q}/\langle a^{kp}\rangle$. We write $u$ and $v$ for the images of $a$ and $b$. Then the relation $b^{kq}=(b^{q})^{k}=(a^{p})^{k}=1$ holds in $\G_{k}$, and
    \[\G_{k}\cong \langle u,v\mid u^{kp}=v^{kq}=1, \; u^{p}=v^{q}\rangle\cong \Z/(kp) *_{\Z/k} \Z/(kq),\]
    the amalgamated subgroup being generated by the image of $a^{p}$, which has order $k$.
    \begin{theorem}\label{Thm:tknot}
        Let $p,q,k\in\N$ with $k\ge1$, $p,q\ge2$ and $\gcd(p,q)=1$. Then the trivial module $\Z\in\lMod{\ZGk}$ has a $2$-periodic free resolution in degrees $\ge 1$ of the form
        \[
            \begin{tikzcd}
                \cdots \arrow[r, "\partial_{2}"] & \ZGk^{2} \arrow[r, "\partial_{3}"] & \ZGk^{2} \arrow[r, "\partial_{2}"] & \ZGk^{2} \arrow[r, "\partial_{1}"] & \ZGk \arrow[r, "\varepsilon"] & \Z \arrow[r] & 0,
            \end{tikzcd}
        \]
        with differentials given by
        \[\partial_{1}=\begin{bmatrix}
        u-1 \\
        v-1
        \end{bmatrix},\;\;\partial_{2}=\begin{bmatrix}
        	    \Ngk{u}{kp-1} & 0 \\
                \Ngk{u}{p-1} & -\Ngk{v}{q-1}
        	\end{bmatrix},\;\;\partial_{3} =\begin{bmatrix}
        	    u-1 & 0\\
                v-1 & (1-v)\Ngk{v^{q}}{k-1}
        	\end{bmatrix}.\]
    \end{theorem}
    \begin{proof}
        Write $G=G_{p,q}$ and $x=1-a^{kp}$. By applying Theorem~\ref{Thm:preres} to the presentation with the single relator $\rho=a^{p}b^{-q}$ we obtain
        \[d_{1}=\begin{bmatrix}
            a-1 \\
            b-1
        \end{bmatrix},\qquad d_{2} = \begin{bmatrix}
            \om^{*}\left(\Fox{\rho}{a}\right) & \om^{*}\left(\Fox{\rho}{b}\right)
        \end{bmatrix}= \begin{bmatrix}
            \Ngk{a}{p-1} & -\Ngk{b}{q-1}
        \end{bmatrix}.\]
        Since $\rho$ is not a proper power in the free group on $a$ and $b$, it follows from Lyndon's identity theorem that $d_{2}$ is injective \cite[\S11, p.~663]{lyndon}. Then
        \[
        	\begin{tikzcd}
        		0 \arrow[r] & \ZG \arrow[r,"d_{2}"] & \ZG^{2} \arrow[r, "d_{1}"] & \ZG \arrow[r, "\varepsilon"] & \Z \arrow[r] & 0
        	\end{tikzcd}
        \]
        is a finite type free resolution $\bul{F}\to\Z$, and in particular $G$ is of type $\FL$ with $\chi(G)=1-2+1=0$ and $\cd(G)\le2$. A free group with non-trivial center is infinite cyclic. Since the center of $G$ is non-trivial and $G/\langle a^{p}\rangle\cong\Z/p*\Z/q$ is non-abelian, it follows that $G$ is not free, and hence $\cd(G)=2$ by the theorem of Stallings and Swan \cite{stall,swancd1}.

        By Theorem~\ref{Thm:zg} it suffices to find a null-homotopy for multiplication by $x$ on $\bul{F}$, and one is given by
        \begin{align*}
            \varphi_{1} =\begin{bmatrix}
                -\Ngk{a}{kp-1} & 0
            \end{bmatrix}, \;\;\;
            \varphi_{2} = \begin{bmatrix}
                0 \\
                (b-1)\Ngk{b^{q}}{k-1}
            \end{bmatrix}.
        \end{align*}
        Since $\bul{F}$ has length $2$, it follows from Proposition~\ref{Prop:starlow}(1) that $\varphi^{2}=0$ over $\ZG$, and hence by Theorem~\ref{Thm:main} the differentials are given by
        \[\partial_{1}=\od_{1},\qquad\partial_{2}=\left[\,\begin{matrix}-\oph_{1}\\ \od_{2}\end{matrix}\,\right],\qquad\partial_{3}=\left[\,\begin{matrix}\od_{1} & -\oph_{2}\end{matrix}\,\right],\]
        with $\partial_{j+2}=\partial_{j}$ for all $j\ge2$.
    \end{proof}

    \begin{corollary}\label{Cor:homfreeprod}
        The homology and cohomology of $\Gamma_{k}$ with integer coefficients are given by
        \[\renewcommand\arraystretch{1.25}
        \begin{array}{r|cc}
            i & \Ho_{i}(\G_{k},\Z) & \Ho^{i}(\G_{k},\Z) \\ \hline
            0 & \Z & \Z \\
            \textup{odd} & \Z/(kpq) & 0 \\
            \textup{even},\ i>0 & 0 & \Z/(kpq)
        \end{array}\]
    \end{corollary}
    \begin{proof}
        Applying the augmentation map $\varepsilon:\Z\G_{k}\to\Z$ to the entries of the differentials of Theorem~\ref{Thm:tknot}, we obtain the following integer matrices and their Smith normal forms:
        \[\varepsilon(\partial_{1})=\begin{bmatrix}
        0 \\
        0
        \end{bmatrix},\;\;\varepsilon(\partial_{2})=\begin{bmatrix}
        	    kp & 0 \\
                p & -q
        	\end{bmatrix}\sim\begin{bmatrix}
        	    1 & 0 \\
                0 & kpq
        	\end{bmatrix},\;\;\varepsilon(\partial_{3}) =\begin{bmatrix}
        	    0 & 0\\
                0 & 0
        	\end{bmatrix}.\]
    \end{proof}
    Since the groups $\G_{k}$ act on a tree with finite stabilizers, they can also be studied by Bass--Serre theory, as in the first chapter of \cite{trees}.

    \begin{example}
        Taking $p=2$, $q=3$ and $k=2$ in Theorem~\ref{Thm:tknot}, we obtain a resolution for $\G=\SLZ\cong \Z/4 *_{\Z/2} \Z/6$, presented as $\langle u,v\mid u^{4}=v^{6}=1,\; u^{2}=v^{3}\rangle$, where $u$ and $v$ correspond to the standard generators $\begin{psmallmatrix}0 & -1\\ 1 & 0\end{psmallmatrix}$ and $\begin{psmallmatrix}0 & -1\\ 1 & 1\end{psmallmatrix}$ of $\SLZ$. The differentials of this resolution are given by
        \[\partial_{1}=\begin{bmatrix}
        u-1 \\
        v-1
        \end{bmatrix},\;\;\partial_{2}=\begin{bmatrix}
        	    1+u+u^{2}+u^{3} & 0 \\
                1+u & -(1+v+v^{2})
        	\end{bmatrix},\;\;\partial_{3} =\begin{bmatrix}
        	    u-1 & 0\\
                v-1 & (1-v)(1+v^{3})
        	\end{bmatrix}.\]
        
        This has the same free rank in each degree as the resolution of Bui and Ellis \cite[\S3]{bui}, and Corollary~\ref{Cor:homfreeprod} recovers the integral (co)homology of $\SLZ$ (see \cite[Ex.~3.3.5, pp.238--241]{ellis}).
    \end{example}

    \begin{remark}\label{Rem:manycyclics}
        The same argument applies to an amalgamated product of finitely many cyclic groups over a common subgroup. Let $\ell\ge2$ and $n_{1},\ldots,n_{\ell}\ge2$, and consider the group
        \[G=\langle a_{1},\ldots,a_{\ell}\mid a_{1}^{n_{1}}=\cdots=a_{\ell}^{n_{\ell}}\rangle,\]
        whose amalgamated subgroup is generated by the central element $c=a_{1}^{n_{1}}$. By Lyndon's identity theorem \cite[pp.~658--659]{lyndon}, the Fox derivatives of the relators $a_{t}^{n_{t}}a_{t+1}^{-n_{t+1}}$ for $1\le t<\ell$ give a finite type free resolution
        \[0\longrightarrow\ZG^{\ell-1}\longrightarrow\ZG^{\ell}\longrightarrow\ZG\longrightarrow\Z\longrightarrow0,\]
        so that $G$ is of type $\FL$ with $\chi(G)=0$ and $\cd(G)\le2$. Moreover, since the quotient $G/\langle c\rangle\cong\Z/n_{1}*\cdots*\Z/n_{\ell}$ by the non-trivial central subgroup $\langle c\rangle$ is non-abelian, the group $G$ is not free, and hence $\cd(G)=2$ by the theorem of Stallings and Swan \cite{stall,swancd1}. It follows that for all $k\ge1$ Theorem~\ref{Thm:zg} applies to the quotient
        \[G/\langle c^{k}\rangle\cong\Z/(kn_{1}) *_{\Z/k}\cdots *_{\Z/k} \Z/(kn_{\ell}),\]
        and gives a free resolution which is $2$-periodic in degrees $\ge1$, with stable rank $\ell$.
    \end{remark}

    \subsection{Central quotients of Heisenberg groups}
    \begin{definition}
        The \emph{integral Heisenberg group} is the finitely presented group 
        \[H=H_{3}(\Z)=\langle a,b,c\mid [a,b]=c,\; [a,c]=1,\; [b,c]=1\rangle,\]
        that is, the group of upper unitriangular $3\times3$ integer matrices.
    \end{definition}
    The center of $H$ is infinite cyclic generated by $c$, and the central quotient is 
    \[H/Z(H)\cong\langle a,b,c\mid [a,b]=c=1,\; [a,c]=1,\; [b,c]=1\rangle\cong \langle a,b\mid [a,b]=1\rangle\cong \Z^{2}.\] 
    The central series $1\trianglelefteq Z(H) \trianglelefteq H$ has free abelian quotients of ranks $1$ and $2$, so $H$ is a finitely generated torsion-free nilpotent group of Hirsch length $3$ \cite[\S7.3]{bieri}. Then by \cite[Ex.~9.8.1, p.166]{bieri} the group $H$ is of type $\FL$ with $\cd(H)=3$.    
    
    A finite type free resolution for the trivial module $\Z\in\ZHmod$ is given in \cite[\S3]{h3res}:
    \begin{equation}\label{Eq:hbergres}
        F_{\bullet}=	\begin{tikzcd}
            	0 \arrow[r] & \ZH \arrow[r,"d_{3}"] & \ZH^{3} \arrow[r,"d_{2}"] & \ZH^{3} \arrow[r, "d_{1}"] & \ZH \arrow[r, "\varepsilon"] & \Z \arrow[r] & 0
            \end{tikzcd}
    \end{equation}
    with differentials given by
    \[d_{1}=\begin{bmatrix}
                a-1 \\
                b-1 \\
                c-1
        \end{bmatrix},\;\; d_{2} = \begin{bmatrix}
                1-cb & a-c & -1 \\
                1-c & 0 & a-1 \\
                0 & 1-c & b-1
            \end{bmatrix},\;\;d_{3} = \begin{bmatrix}
            c-1 & 1-cb & a-c
        \end{bmatrix}.
    \]
    \begin{remark}
        Majumdar and Sultana used this resolution to compute the homology and cohomology of $H_{3}(\Z)$ \cite[\S\S4 and 5]{h3res}. More generally, Lee and Packer computed the cohomology of the higher integral Heisenberg groups \cite[Thm.~1.8, p.231]{leepacker}.
    \end{remark}
    
    Now, let $k\in\N$ with $k\ge1$ and consider the central subgroup $\langle c^{k}\rangle\subseteq Z(H)$. We apply Theorem~\ref{Thm:zg} to the quotient
    \[\Hil_{k}=H/\langle c^{k}\rangle\cong\langle a,b,c\mid [a,b]=c,\; [a,c]=[b,c]=c^{k}=1\rangle,\]
    starting with the finite type free resolution \eqref{Eq:hbergres}.
    \begin{theorem}\label{Thm:hberg}
        Let $k\in\N$ with $k\ge1$. Then the trivial module $\Z\in\lMod{\Z\Hil_{k}}$ has a $2$-periodic free resolution in degrees $\ge 2$ of the form
        \[
            \begin{tikzcd}
                \cdots \arrow[r, "\partial_{3}"] & \Z\Hil_{k}^{4} \arrow[r, "\partial_{4}"] & \Z\Hil_{k}^{4}  \arrow[r, "\partial_{3}"] & \Z\Hil_{k}^{4}  \arrow[r, "\partial_{2}"] & \Z\Hil_{k}^{3}  \arrow[r, "\partial_{1}"] & \Z\Hil_{k} \arrow[r, "\varepsilon"] & \Z \arrow[r] & 0
            \end{tikzcd}
        \]
        with the differentials given by
        \begin{align*}
            \partial_{1}=\begin{bmatrix}
            a-1 \\
            b-1 \\
            c-1
            \end{bmatrix},\;\; \partial_{2}&=\begin{bmatrix}
                	    0 & 0 & \Ngk{c}{k-1}  \\
                        1-cb & a-c & -1\\
                        1-c & 0 & a-1\\
                        0 & 1-c & b-1
                	\end{bmatrix},\;\; \partial_{3} =\begin{bmatrix}
                	    a-1 & 0 & -\Ngk{c}{k-1} & 0\\
                        b-1 & 0 & 0 & -\Ngk{c}{k-1} \\
                        c-1 & 0 & 0 & 0 \\
                        0 & c-1 & 1-cb & a-c
                	\end{bmatrix}, \\
                    \partial_{4}&=\begin{bmatrix}
                	    0 & 0 & \Ngk{c}{k-1} & 0\\
                        1-cb & a-c & -1 & \Ngk{c}{k-1} \\
                        1-c & 0 & a-1 & 0 \\
                        0 & 1-c & b-1 & 0
                	\end{bmatrix}.
        \end{align*}
    \end{theorem}
    \begin{proof}
        By Theorem~\ref{Thm:zg} it suffices to find a null-homotopy for multiplication by $x=1-c^{k}$ on \eqref{Eq:hbergres}, and one is given by
        \begin{align*}
            \varphi_{1} = \begin{bmatrix}
                0 & 0 & -\Ngk{c}{k-1}
            \end{bmatrix},\;\;\;
            \varphi_{2}=\begin{bmatrix}
                        0 & \Ngk{c}{k-1} & 0\\
                        0 & 0 & \Ngk{c}{k-1} \\
                        0 & 0 & 0
                \end{bmatrix}, \;\;\; \varphi_{3} = \begin{bmatrix}
                    -\Ngk{c}{k-1} \\
                                0\\
                                0
                \end{bmatrix}.
        \end{align*}
        The only non-zero entry of $\varphi_{1}$ lies in its third column and the third row of $\varphi_{2}$ is zero, and hence $\varphi_{1}\varphi_{2}=0$ over $\ZH$. Then since \eqref{Eq:hbergres} has length $3$, it follows from Proposition~\ref{Prop:starlow}(2) that $\varphi_{2}\varphi_{3}=0$ as well, so that $\varphi^{2}=0$, and by Theorem~\ref{Thm:main} the differentials are given by
        \begin{equation}\label{Eq:blocksthree}
            \partial_{1}=\od_{1},\qquad\partial_{2}=\left[\,\begin{matrix}-\oph_{1}\\ \od_{2}\end{matrix}\,\right],\qquad\partial_{3}=\left[\,\begin{matrix}\od_{1} & -\oph_{2}\\ 0 & \od_{3}\end{matrix}\,\right],\qquad\partial_{4}=\left[\,\begin{matrix}-\oph_{1} & 0\\ \od_{2} & -\oph_{3}\end{matrix}\,\right],
        \end{equation}
        with $\partial_{j+2}=\partial_{j}$ for all $j\ge3$.
    \end{proof}

    \begin{corollary}
        The homology and cohomology of $\Hil_{k}$ with integer coefficients are given by
        \[\renewcommand\arraystretch{1.25}
        \begin{array}{r|cc}
            i & \Ho_{i}(\Hil_{k},\Z) & \Ho^{i}(\Hil_{k},\Z) \\ \hline
            0 & \Z & \Z \\
            1 & \Z^{2} & \Z^{2} \\
            2 & \left(\Z/k\right)^{2}\oplus \Z & \Z \\
            \textup{odd},\ i>1 & \Z/k^{2} & \left(\Z/k\right)^{2} \\
            \textup{even},\ i>2 & \left(\Z/k\right)^{2} & \Z/k^{2}
        \end{array}\]
    \end{corollary}
    \begin{proof}
        Applying the augmentation map to the entries of the differentials of Theorem~\ref{Thm:hberg}, we obtain the following integer matrices and their Smith normal forms:
        \begin{align*}
            &\varepsilon(\partial_{1})=\begin{bmatrix}
                0 \\
                0 \\
                0
            \end{bmatrix}, &&\varepsilon(\partial_{2})=\begin{bmatrix}
                0 & 0 & k \\
                0 & 0 & -1\\
                0 & 0 & 0 \\
                0 & 0 & 0
            \end{bmatrix}\sim\begin{bmatrix}
                1 & 0 & 0 \\
                0 & 0 & 0 \\
                0 & 0 & 0 \\
                0 & 0 & 0
            \end{bmatrix}, \\
            &\varepsilon(\partial_{3})= \begin{bmatrix}
            	    0 & 0 & -k & 0\\
                    0 & 0 & 0 & -k \\
                    0 & 0 & 0 & 0 \\
                    0 & 0 & 0 & 0
            	\end{bmatrix}\sim \begin{bmatrix}
            	    k & 0 & 0 & 0\\
                    0 & k & 0 & 0 \\
                    0 & 0 & 0 & 0 \\
                    0 & 0 & 0 & 0
            	\end{bmatrix}, &&\varepsilon(\partial_{4}) =\begin{bmatrix}
            	    0 & 0 & k & 0\\
                    0 & 0 & -1 & k \\
                    0 & 0 & 0 & 0 \\
                    0 & 0 & 0 & 0
            	\end{bmatrix} \sim \begin{bmatrix}
            	    1 & 0 & 0 & 0\\
                    0 & k^{2} & 0 & 0 \\
                    0 & 0 & 0 & 0 \\
                    0 & 0 & 0 & 0
            	\end{bmatrix}.
        \end{align*}
    \end{proof}

    \subsection{Hyperbolic triangle groups}\label{ss:triangle}
    Let $p,q,r\in\N$ with $2\le p\le q\le r$ and $\frac{1}{p}+\frac{1}{q}+\frac{1}{r}<1$.
    \begin{definition}
        The $(p,q,r)$\emph{-hyperbolic triangle group} is the finitely presented group
        \[T=T(p,q,r)=\langle a,b\mid a^{p}=b^{q}=(ab)^{r}=1\rangle.\]
    \end{definition}
    By \cite[Thm.~II.5.2.8, p.81]{magnustess} the group $T(p,q,r)$ is the von Dyck subgroup, of index $2$, in the group of isometries of the hyperbolic plane generated by the reflections in the edges of a triangle with angles $\frac{\pi}{p},\frac{\pi}{q},\frac{\pi}{r}$.

    Consider the group $G=G(p,q,r)=\langle a,b\mid a^{p}=b^{q}=(ab)^{r}\rangle$. The element $c=a^{p}=b^{q}$ commutes with $a$ and with $b$, and hence it is central in $G$, with $G/\langle c\rangle\cong T(p,q,r)$ (for central extensions of the triangle groups see \cite[\S3]{milnorpqr}). Since $T$ is infinite, so is $G$, and by a theorem of Strebel \cite[Thm.~A(ii)]{streb} the trivial module has a finite type free resolution $\bul{F}\to\Z\in\ZGmod$ of the form
    \begin{equation}\label{Eq:Trires}
        F_{\bullet}=	\begin{tikzcd}
            		0 \arrow[r] & \ZG \arrow[r,"d_{3}"] & \ZG^{2} \arrow[r,"d_{2}"] & \ZG^{2} \arrow[r, "d_{1}"] & \ZG \arrow[r, "\varepsilon"] & \Z \arrow[r] & 0,
            	\end{tikzcd}
    \end{equation}
    where the differentials are given by
    \[
        d_{1}=\begin{bmatrix}
            a-1\\[2pt]
            b-1
        \end{bmatrix},\;\; d_{2}= \begin{bmatrix}
            \Ngk{ab}{r-1} & \Ngk{ab}{r-1}\cdot a - \Ngk{b}{q-1}\\[2pt]
            \Ngk{ba}{r-1}\cdot b - \Ngk{a}{p-1} & \Ngk{ba}{r-1}
        \end{bmatrix},\;\; d_{3} = \begin{bmatrix}
            1-b & 1-a
        \end{bmatrix}.
    \]
    In particular $G$ is of type $\FL$, and by the same theorem $G$ is an orientable Poincar\'e duality group of dimension $3$, so that $\cd(G)=3$.

    \begin{theorem}\label{Thm:tri}
        The trivial module $\Z\in\lMod{\Z T}$ has a $2$-periodic free resolution in degrees $\ge 2$ of the form
        \[
            \begin{tikzcd}
                \cdots \arrow[r, "\partial_{3}"] & \Z T^{3} \arrow[r, "\partial_{4}"] & \Z T^{3}  \arrow[r, "\partial_{3}"] & \Z T^{3}  \arrow[r, "\partial_{2}"] & \Z T^{2}  \arrow[r, "\partial_{1}"] & \Z T \arrow[r, "\varepsilon"] & \Z \arrow[r] & 0,
            \end{tikzcd}
        \]
        with the differentials given by
        \begin{align*}
            \partial_{1}=\begin{bmatrix}
            a-1 \\
            b-1 
            \end{bmatrix},\;\; \partial_{2}&=\begin{bmatrix}
                	    \Ngk{a}{p-1} & 0 \\
                        \Ngk{ab}{r-1} & \Ngk{ab}{r-1}\cdot a - \Ngk{b}{q-1}\\
                        \Ngk{ba}{r-1}\cdot b - \Ngk{a}{p-1} & \Ngk{ba}{r-1}
                	\end{bmatrix},\;\; \partial_{3} =\begin{bmatrix}
                	    a-1 & 0 & 0 \\
                        b-1 & 0 & ba-1\\
                        0 & 1-b & 1-a
                	\end{bmatrix}, \\
                    \partial_{4}&=\begin{bmatrix}
                	    \Ngk{a}{p-1} & 0 & 0\\
                        \Ngk{ab}{r-1} & \Ngk{ab}{r-1}\cdot a-\Ngk{b}{q-1} & -\Ngk{b}{q-1} \\
                        \Ngk{ba}{r-1}\cdot b-\Ngk{a}{p-1} & \Ngk{ba}{r-1} & 0
                	\end{bmatrix}.
        \end{align*}
    \end{theorem}
    \begin{proof}
        By Theorem~\ref{Thm:zg} it suffices to find a null-homotopy for multiplication by $x=1-a^{p}$ on the resolution \eqref{Eq:Trires}, and one is given by
        \begin{align*}
            \varphi_{1} = \begin{bmatrix}
            -\Ngk{a}{p-1} & 0
        \end{bmatrix},\;\;
            \varphi_{2}=\begin{bmatrix}
                            0 & 0 \\
                            0 & 1-ba
                             \end{bmatrix},\;\;
            \varphi_{3}=\begin{bmatrix}
                            \Ngk{b}{q-1} \\
                            0
                             \end{bmatrix}.
        \end{align*}
        The first row of $\varphi_{2}$ is zero, and hence $\varphi_{1}\varphi_{2}=0$ over $\ZG$. Then since \eqref{Eq:Trires} has length $3$, it follows from Proposition~\ref{Prop:starlow}(2) that $\varphi_{2}\varphi_{3}=0$ as well, so that $\varphi^{2}=0$, and by Theorem~\ref{Thm:main} the differentials are given by the matrices \eqref{Eq:blocksthree} in the proof of Theorem~\ref{Thm:hberg}.
    \end{proof}

    \begin{corollary}\label{Cor:trihomology}
        The homology and cohomology of $T=T(p,q,r)$ with integer coefficients are given by
        \[\renewcommand\arraystretch{1.25}
        \begin{array}{r|cc}
            i & \Ho_{i}(T,\Z) & \Ho^{i}(T,\Z) \\ \hline
            0 & \Z & \Z \\
            1 & \Z/\alpha\oplus\Z/\beta & 0 \\
            2 & \Z & \Z/\alpha\oplus\Z/\beta\oplus\Z \\
            \textup{odd},\ i>1 & \Z/p\oplus\Z/q\oplus\Z/r & 0 \\
            \textup{even},\ i>2 & 0 & \Z/p\oplus\Z/q\oplus\Z/r
        \end{array}\]
        where $\alpha=\gcd\{p,q,r\}$ and $\alpha\beta=\gcd\{pq,pr,qr\}$.
    \end{corollary}
    \begin{proof}
        Applying the augmentation map to the entries of the differentials of Theorem~\ref{Thm:tri}, we obtain the following integer matrices:
        \begin{align*}
            &\varepsilon(\partial_{1})=\begin{bmatrix}
                0 \\
                0 
            \end{bmatrix}, &&\varepsilon(\partial_{2})=\begin{bmatrix}
                p & 0 \\
                r & r-q \\
                r-p & r
            \end{bmatrix}, \\
            &\varepsilon(\partial_{3})=\begin{bmatrix}
                0 & 0 & 0 \\
                0 & 0 & 0 \\
                0 & 0 & 0
            \end{bmatrix}, &&\varepsilon(\partial_{4})  = \begin{bmatrix}
                p & 0 & 0 \\
                r & r-q & -q \\
                r-p & r & 0 
            \end{bmatrix}.
        \end{align*}
        The entries of $\varepsilon(\partial_{2})$ have greatest common divisor $\gcd\{p,q,r\}=\alpha$. Its $2\times2$ minors are $pr-pq$, $pr$ and $pr+qr-pq$, which generate the same ideal of $\Z$ as $pq$, $pr$ and $qr$, and hence their greatest common divisor is $\gcd\{pq,pr,qr\}=\alpha\beta$. It follows that the Smith normal form of $\varepsilon(\partial_{2})$ is
        \[\varepsilon(\partial_{2})\sim\begin{bmatrix}
        \alpha & 0 \\
        0 & \beta \\
        0 & 0
        \end{bmatrix},\]
        which gives the homology and cohomology in degrees $i\le2$. By elementary row and column operations we also have
        \[\varepsilon(\partial_{4})\sim\begin{bmatrix}
        p & 0 & 0 \\
        0 & q & 0 \\
        0 & 0 & r
        \end{bmatrix},\]
        which gives the homology and cohomology in degrees $i\ge3$.
    \end{proof}
    \begin{remark}
        Corollary~\ref{Cor:trihomology} recovers the well-known result for the homology of hyperbolic triangle groups $T=T(p,q,r)$:
        \[\Ho_{i}(T,\Z)\cong \Ho_{i}(\Z/p, \Z)\oplus \Ho_{i}(\Z/q, \Z)\oplus \Ho_{i}(\Z/r, \Z)\]
        for all $i\ge 3$ (see \cite[Prop.~6.4.1, p.442]{ellis}).
    \end{remark}
    Ellis and Williams \cite{ellistri} construct free resolutions for $T(p,q,r)$ and, as a digression, for infinite cyclic central extensions of the triangle groups which include the groups $G(p,q,r)$ \cite[\S\S3--4]{ellistri}. Note that in their notation $G(l,m,n)$ refers to generalized triangle groups and differs from ours. The resolution of $T(p,q,r)$ in \cite{ellistri}, which is obtained by the perturbation method recalled in \S\ref{ss:earlier}, has ranks $3$, $5$, $4$ in degrees $0$, $1$, $2$ and rank $3$ in all higher degrees \cite[Eq.~(4), p.576]{ellistri}. The resolution of Theorem~\ref{Thm:tri} has ranks $1$ and $2$ in degrees $0$ and $1$, and rank $3$ in every higher degree. 

    \subsection{Artin groups of spherical type}\label{ss:artin}
    The last family of groups we consider is the class of central quotients of Artin groups of spherical type. This class contains the central quotients of the braid groups, which are the mapping class groups of the punctured plane.

    \begin{definition}
        Let $(W,\mathcal{S})$ be a Coxeter system of rank $s=\vert \mathcal{S}\vert$ with Coxeter matrix
        $(m_{\sigma\sigma'})_{\sigma,\sigma'\in\mathcal{S}}$. The associated \emph{Artin group} (see \cite{BrieskornSaito}) is
        \[
            \mathcal{A}=\mathcal{A}(W,\mathcal{S})=\left\langle \mathcal{S} \;\middle|\; \underbrace{\sigma\sigma'\sigma\cdots}_{m_{\sigma\sigma'}} =\underbrace{\sigma'\sigma\sigma'\cdots}_{m_{\sigma\sigma'}} \;\text{ for }\sigma\neq\sigma'\in\mathcal{S},\; m_{\sigma\sigma'}<\infty\right\rangle.\]
        We call $\mathcal{A}$ of \emph{spherical type} when the Coxeter group $W$ is finite.
    \end{definition}

    The braid group $\Br_{m+1}$ on $m+1$ strands, introduced by Artin \cite{artinbraid}, is the Artin group of type $A_{m}$, presented on $\sigma_{1},\ldots,\sigma_{m}$ by the relations $\sigma_{i}\sigma_{i+1}\sigma_{i}=\sigma_{i+1}\sigma_{i}\sigma_{i+1}$ and $\sigma_{i}\sigma_{j}=\sigma_{j}\sigma_{i}$ for $\vert i-j\vert\ge2$.

    To see that $\mathcal{A}$ is of type $\FL$, and to compute its cohomological dimension, we use the topology of a hyperplane complement. The group $\mathcal{A}$ is the fundamental group of the quotient by $W$ of the complement in $\C^{s}$ of the complexified reflection hyperplanes of $W$, and by a theorem of Deligne this complement is aspherical \cite{delignebraid}. By \cite[Cor.~1.6, p.570]{salvetti} the Salvetti complex is a finite $K(\mathcal{A},1)$ with one cell in dimension $i$ for each $i$-element subset of $\mathcal{S}$. Then the universal cover of the Salvetti complex is a contractible free $\mathcal{A}$-CW-complex of dimension $s$, and Squier showed that $\mathcal{A}$ is a duality group of dimension $s$ \cite[Thms.~A and B]{squier}. Together these give the following.
    \begin{theorem}[Deligne, Salvetti, Squier]\label{Thm:cdartin}
        Let $\mathcal{A}$ be an Artin group of spherical type and rank $s$. Then $\mathcal{A}$ is torsion-free and of type $\FL$ with $\cd(\mathcal{A})=s$, and the cellular chain complex of the universal cover of the Salvetti complex is a finite type free resolution $\bul{F}\to\Z$ over $\Z\mathcal{A}$ with
        \[\rank_{\Z\mathcal{A}}(F_{i})=\binom{s}{i}\qquad(0\le i\le s).\]
    \end{theorem}
    In type $A_{m}$ this recovers the resolution coming from the free $\Br_{m+1}$-CW structure on the permutahedron \cite[\S4.1.5]{loday}, in which $F_{i}$ is free on the tuples $[a_{1},\ldots,a_{m+1-i}]$ of positive integers summing to $m+1$. A stars-and-bars count gives $\binom{m}{i}$ again.

    The center of $\mathcal{A}$ was determined by Brieskorn--Saito \cite[\S7]{BrieskornSaito} and Deligne \cite[\S4]{delignebraid}, and for the braid groups earlier by Chow \cite[Thm.~III]{chow}. See also \cite[Thm.~7, p.246]{garside} and \cite[\S1.3]{bestar}.
    \begin{theorem}\label{Thm:artincentre}
        Let $\mathcal{A}$ be an irreducible Artin group of spherical type and rank $s\ge2$, with Garside element $\Delta$, the positive lift of the longest element of $W$ (see \cite{BrieskornSaito}). Then $Z(\mathcal{A})$ is infinite cyclic, generated by $\Delta$ or by $\Delta^{2}$ according to the type. For the braid group $\Br_{m+1}$, of type $A_{m}$, the generator is the \emph{full twist} $\Delta^{2}=\left(\sigma_{1}\sigma_{2}\cdots\sigma_{m}\right)^{m+1}$.
    \end{theorem}
    The argument below does not depend on which of the two generates the center. We write $c$ for a generator of $Z(\mathcal{A})$ and $\G_{\mathcal{A}}=\mathcal{A}/\langle c\rangle=\mathcal{A}/Z(\mathcal{A})$. Then there is a central extension
    \begin{equation}\label{Eq:mcgplane}
        \begin{tikzcd}
			1 \arrow[r] & \langle c\rangle \arrow[r] & \mathcal{A} \arrow[r] & \G_{\mathcal{A}} \arrow[r] & 1.
		\end{tikzcd}
     \end{equation}
    By \cite{bestar} the group $\G_{\mathcal{A}}$ acts cocompactly with finite stabilizers on a contractible complex, and it is a virtual duality group of dimension $s-1$ by \cite[Thm.~3.6]{bestar}, which recovers a theorem of Squier \cite[Thms.~A and B]{squier}. In particular $\vcd(\G_{\mathcal{A}})=s-1$ in the sense of Definition~\ref{Def:FP}.

    \begin{theorem}\label{Thm:mcgperiod}
        Let $\mathcal{A}$ be an irreducible Artin group of spherical type and rank $s\ge2$, and let $\G_{\mathcal{A}}=\mathcal{A}/Z(\mathcal{A})$. Then the trivial module $\Z\in\lMod{\Z\G_{\mathcal{A}}}$ has a finite type free resolution which is $2$-periodic in degrees $\ge s-1$, and the stable rank in the periodic part is $2^{s-1}$.
    \end{theorem}
    \begin{proof}
        By Theorem~\ref{Thm:cdartin} the group $\mathcal{A}$ is of type $\FL$ with $\pd_{\Z\mathcal{A}}(\Z)=\cd(\mathcal{A})=s$, and by Theorem~\ref{Thm:artincentre} the center $Z(\mathcal{A})=\langle c\rangle$ is infinite cyclic, so $\chi(\mathcal{A})=0$ by Theorem~\ref{Thm:gottlieb}. Then Theorem~\ref{Thm:zg} applies to \eqref{Eq:mcgplane} and gives a resolution which is $2$-periodic in degrees $\ge s-1$, of stable rank
        \[\sum_{i\equiv 0\!\!\pmod{2}}\binom{s}{i}=2^{s-1}=\sum_{i\equiv 1\!\!\pmod{2}}\binom{s}{i}. \qedhere\]
    \end{proof}

    When $\mathcal{A}$ is of type $A_{m}$, Theorem~\ref{Thm:mcgperiod} gives resolutions over the central quotients of the braid groups, which are the mapping class groups of the punctured plane. To see this, we identify $\Br_{m+1}$ with the mapping class group of the closed $2$-disk in $\R^{2}$ with $m+1$ punctures, each generator $\sigma_{i}$ going to a \emph{half-twist} about two punctures as in \cite[Fig.~9.5, p.244]{primer}. Then the central quotient $\Br_{m+1}/Z(\Br_{m+1})$ is the mapping class group of the $(m+1)$-times punctured plane $\mathbb{P}_{m+1}$ \cite[\S2.2, pp.28--29]{braid}.
    \begin{corollary}\label{Cor:mcgperiod}
        For each $m\ge2$ the mapping class group $\mathrm{Mod}(\mathbb{P}_{m+1})\cong \Br_{m+1}/Z(\Br_{m+1})$ has a $2$-periodic free resolution in degrees $\ge m-1$, of stable rank $2^{m-1}$.
    \end{corollary}
    These groups sit inside the mapping class groups of punctured spheres. By \cite[p.248]{primer} the group $\mathrm{Mod}(\mathbb{P}_{m+1})$ has index $m+2$ in the mapping class group $\mathrm{Mod}(\mathbb{S}_{m+2})$ of the $(m+2)$-times punctured sphere $\mathbb{S}_{m+2}$, which has virtual cohomological dimension $m-1$ by the work of Harer \cite[Thm.~4.1]{harer}. Benson and Cohen showed in \cite[Thm.~1.1, p.6]{bensoncohen} that in positive degrees $\Ho_{\ast}(\mathrm{Mod}(\mathbb{S}_{m+2}),\Z)$ is a torsion group, with $p$-torsion only for primes $p\le m+2$.

    \begin{example}\label{Ex:braid4}
        We compute the case $s=3$ of Theorem~\ref{Thm:mcgperiod} in type $A_{3}$, that is the case $m=3$ of Corollary~\ref{Cor:mcgperiod}.

        The braid group on four strands is the finitely presented group
        \[\Br_{4}=\langle \so,\st,\sh\mid \so\st\so=\st\so\st,\; \st\sh\st=\sh\st\sh,\; \so\sh=\sh\so\rangle.\]
        Let $\gamma=\so\st\sh$, so that $\gamma^{4}=\Delta^{2}$ is the full twist of Theorem~\ref{Thm:artincentre}. Then $\gamma^{4}$ generates the center $\langle \gamma^{4} \rangle = Z(\Br_{4})\trianglelefteq \Br_{4}$, and as in Section~\ref{S:applications} we take $c=\gamma^{4}$ and $x=1-\gamma^{4}\in\ZB_{4}$. The trivial module $\Z\in\lMod{\ZB_{4}}$ has a resolution $\bul{F}$ of the form
        \[
            \begin{tikzcd}
    			0 \arrow[r] & \ZB_{4} \arrow[r, "d_{3}"] & \ZB_{4}^3 \arrow[r, "d_{2}"] & \ZB_{4}^{3} \arrow[r, "d_{1}"] & \ZB_{4} \arrow[r, "\varepsilon"] & \Z \arrow[r] & 0
    		\end{tikzcd}
         \]
        where the first two differentials are found as in Theorem~\ref{Thm:preres}
        \[d_{1}=\begin{bmatrix}
            \so-1\\
            \st-1\\
            \sh-1
        \end{bmatrix},\;\; d_{2}=\begin{bmatrix}
            1+\so\st-\st & \so-\st\so-1 & 0\\
            0 & 1+\st\sh-\sh & \st-\sh\st-1\\
            1-\sh & 0 & \so-1
        \end{bmatrix},\] 
        and the third, $d_{3}$, is constructed from a \emph{homotopical 2-syzygy} in \cite[\S4.1.4, pp.23--24]{loday}
        \[\adjustbox{max width=0.94\textwidth}{$d_{3} = \begin{bmatrix}
            1+\st\sh-\sh-\so\st\sh & 1+\st\so-\so-\sh\st\so & \st+\so\sh\st-1-\sh\st - \so\st - \st\so\sh\st
        \end{bmatrix}$}.\]
        Now, consider the mapping class group 
        \[\G_{\Br_{4}}\coloneqq \langle\so,\st,\sh\mid \so\st\so=\st\so\st,\; \st\sh\st=\sh\st\sh,\; \so\sh=\sh\so,\; \gamma^{4}=1\rangle \cong \Br_{4}/Z(\Br_{4}).\] 
        Then the trivial module $\Z\in\lMod{\Z\G_{\Br_{4}}}$ has a $2$-periodic free resolution in degrees $\ge 2$ of the form
        \[
            \begin{tikzcd}
                \cdots \arrow[r, "\partial_{3}"] & \Z\G_{\Br_{4}}^{4} \arrow[r, "\partial_{4}"] & \Z\G_{\Br_{4}}^{4}  \arrow[r, "\partial_{3}"] & \Z\G_{\Br_{4}}^{4}  \arrow[r, "\partial_{2}"] & \Z\G_{\Br_{4}}^{3}  \arrow[r, "\partial_{1}"] & \Z\G_{\Br_{4}} \arrow[r, "\varepsilon"] & \Z \arrow[r] & 0
            \end{tikzcd}
        \]
        with the differentials given by
            \begin{align*}
                \partial_{1}&=\begin{bmatrix}
                    \so-1 \\
                    \st-1 \\
                    \sh-1
                \end{bmatrix},\\
                \partial_{2}&=\begin{bmatrix}
            	    \Ngk{\gamma}{3} & \Ngk{\gamma}{3}\so & \Ngk{\gamma}{3}\so\st  \\
                    1+\so\st-\st & \so-\st\so-1 & 0\\
                    0 & 1+\st\sh-\sh & \st-\sh\st-1\\
                    1-\sh & 0 & \so-1
            	\end{bmatrix},\\
                \partial_{3} &=\adjustbox{max width=0.94\textwidth}{$\begin{bmatrix}
            	    \so-1 & \gamma^{3} - \so & \gamma^{2}\so - \so\st\so & -\gamma^{3}\so\st + \gamma^{2} + \so^{2}\st -\so\st\\
                    \st-1 & -\gamma\so + 1 & \gamma^{3}\so - \gamma\so\st\so & \gamma^{3} + \gamma\so^{2}\st - \gamma\so\st - \so\st \\
                    \sh-1 & \gamma-\gamma^{2}\so & \so - \gamma^{2}\so\st\so & \gamma^{2}\so^{2}\st - \gamma^{2}\so\st - \gamma\so\st  +1\\
                    0 & 1+\st\sh-\sh-\gamma & 1+\st\so-\so-\sh\st\so & \st+\so\sh\st-1-\sh\st - \so\st - \st\so\sh\st
            	\end{bmatrix}$},\\ 
                \partial_{4}&=\begin{bmatrix}
            	    \Ngk{\gamma}{3} & \Ngk{\gamma}{3}\so & \Ngk{\gamma}{3}\so\st & 0\\
                    1+\so\st-\st & \so - \st\so- 1 & 0 & -\gamma^{3}+\so\st\so \\
                    0 & 1+\st\sh-\sh & \st -\sh\st-1 & -1 + \gamma\so\st\so \\
                    1-\sh & 0 & \so-1 & -\gamma^{2}\so+\so
            	\end{bmatrix}.
            \end{align*}
        By Theorem~\ref{Thm:zg} it suffices to find a null-homotopy for multiplication by $x=1-\gamma^{4}$ on $\bul{F}$, and one is given by
            \begin{align*}
                \varphi_{1} &= \begin{bmatrix}
                    -\Ngk{\gamma}{3} & -\Ngk{\gamma}{3}\so & -\Ngk{\gamma}{3}\so\st
                \end{bmatrix}, \\
                \varphi_{2} &=\adjustbox{max width=0.94\textwidth}{$\begin{bmatrix}
            	     -\gamma^{3}+\so & -\gamma^{2}\so+\so\st\so & \gamma^{3}\so\st - \gamma^{2}-\so^{2}\st+\so\st\\
                     \gamma\so-1 & -\gamma^{3}\so + \gamma\so\st\so & -\gamma^{3} -\gamma\so^{2}\st + \gamma\so\st + \so\st \\
                     \gamma^{2}\so-\gamma & -\so+\gamma^{2}\so\st\so & -\gamma^{2}\so^{2}\st + \gamma^{2}\so\st + \gamma\so\st -1
            	\end{bmatrix}$}, \\ 
                \varphi_{3} &= \begin{bmatrix}
                        \gamma^{3}-\so\st\so \\
                        1-\gamma\so\st\so \\
                        \gamma^{2}\so-\so
                    \end{bmatrix}.
            \end{align*}
        A direct calculation by expanding the products and simplifying with the braid relations shows that $\varphi_{1}\varphi_{2}=0$ over $\ZB_{4}$, and not only modulo $x$. Then since the resolution has length $3$, it follows from Proposition~\ref{Prop:starlow}(2) that $\varphi_{2}\varphi_{3}=0$ as well, so that $\varphi^{2}=0$, and by Theorem~\ref{Thm:main} the differentials are given by the matrices \eqref{Eq:blocksthree} in the proof of Theorem~\ref{Thm:hberg}.
        
        Applying the augmentation map to the entries of these differentials, we obtain the following integer matrices:
            \begin{align*}
                &\varepsilon(\partial_{1})=\begin{bmatrix}
                    0 \\
                    0 \\
                    0
                \end{bmatrix}, &&\varepsilon(\partial_{2})=\begin{bmatrix}
                    4 & 4 & 4 \\
                    1 & -1 & 0 \\
                    0 & 1 & -1 \\
                    0 & 0 & 0
                \end{bmatrix}\sim\begin{bmatrix}
                    1 & 0 & 0 \\
                    0 & 1 & 0 \\
                    0 & 0 & 12 \\
                    0 & 0 & 0
                \end{bmatrix}, \\
                &\varepsilon(\partial_{3})= \begin{bmatrix}
                	    0 & 0 & 0 & 0 \\
                        0 & 0 & 0 & 0 \\
                        0 & 0 & 0 & 0 \\
                        0 & 0 & 0 & -2
                	\end{bmatrix}\sim \begin{bmatrix}
                	    2 & 0 & 0 & 0\\
                        0 & 0 & 0 & 0 \\
                        0 & 0 & 0 & 0 \\
                        0 & 0 & 0 & 0
                	\end{bmatrix}, &&\varepsilon(\partial_{4}) = \begin{bmatrix}
                	    4 & 4 & 4 & 0 \\
                        1 & -1 & 0 & 0 \\
                        0 & 1 & -1 & 0 \\
                        0 & 0 & 0 & 0
                	\end{bmatrix} \sim \begin{bmatrix}
                	    1 & 0 & 0 & 0 \\
                        0 & 1 & 0 & 0 \\
                        0 & 0 & 12 & 0 \\
                        0 & 0 & 0 & 0
                	\end{bmatrix}.
            \end{align*}
        Then the homology and cohomology of $\G_{\Br_{4}}$ with integer coefficients are given by
            \[\renewcommand\arraystretch{1.25}
            \begin{array}{r|cc}
                i & \Ho_{i}(\G_{\Br_{4}},\Z) & \Ho^{i}(\G_{\Br_{4}},\Z) \\ \hline
                0 & \Z & \Z \\
                1 & \Z/12 & 0 \\
                \text{odd},\ i>1 & \Z/12 & \Z/2 \\
                \text{even},\ i>0 & \Z/2 & \Z/12
            \end{array}\]
    \end{example}

    \begin{remark}
        Benson and Cohen described the mod-$2$ cohomology of $\mathrm{Mod}(\mathbb{P}_{m+1})$ by topological methods in \cite[Thm.~1.4, p.7]{bensoncohen}, and Rossi \cite[Thm.~1.1]{rossi} computed the mod-$p$ homology for all primes $p$ and all $m$ using the gravity operad. Example~\ref{Ex:braid4} recovers these results algebraically in the case $m=3$.
    \end{remark}

    \section*{Further directions}
    We close with some questions which are left open by this work.

    It would be useful to have a criterion on the triple $(R,x,N)$ which guarantees the existence of a null-homotopy satisfying $\Phi^{2}=0$. By Proposition~\ref{Prop:starlow} such a null-homotopy always exists when $\pd_{R}(N)\le2$, and Proposition~\ref{Prop:mf} describes the condition in terms of matrix factorizations in general. The differentials of the Salvetti resolutions of \S\ref{ss:artin} are known explicitly by the work of De Concini and Salvetti \cite{dcsalvetti96}, \cite[\S0]{dccsalvetti}, and these resolutions have length equal to the rank of the Artin group, so that they give test cases beyond this range. In Example~\ref{Ex:braid4} we verify the hypothesis in type $A_{3}$, for the resolution of \cite{loday}, and it would be interesting to know whether it holds for all irreducible Artin groups of spherical type.

    Now suppose that $\G$ has finite virtual cohomological dimension. Then the periodicity in Theorem~\ref{Thm:zg}(2) begins as early as possible. Indeed, if $\G$ has periodic cohomology after $k$ steps then $k\ge\vcd(\G)$ by \cite[Thm.~4.9]{mislintalelli}, and $\vcd(\G)\ge n-1$ since the preimage in $G$ of a torsion-free subgroup $\Gp$ of finite index in $\G$ is an extension of $\Gp$ by $\langle c\rangle$ of cohomological dimension $n$ (see \cite[\S VIII.2 and Thm.~VIII.3.1]{brown}). In particular $\vcd(\G)=n-1$ for each of the groups in Section~\ref{S:examples}. For any free resolution of $\Z$ over $\Z\G$, both $\Ho_{i}(\G,\Z)$ and $\Ho^{i}(\G,\Z)$ are subquotients of a free abelian group whose rank is the rank of the resolution in degree $i$. Hence this rank is at least the larger of the minimal numbers of generators of these two groups. For the triangle groups $T(p,q,r)$ with $\gcd\{p,q,r\}>1$ the resolution of Theorem~\ref{Thm:tri} attains this bound in every degree, but for the other families of Section~\ref{S:examples} the resolutions constructed here do not, and we do not know whether eventually periodic resolutions of smaller rank exist.

    The construction also has a geometric interpretation, which is related to Conjecture~\ref{Conj}. If $G$ has a finite $K(G,1)$ with universal cover $X$, then by Remark~\ref{Rem:gysin} the group $\G$ acts freely and cocompactly on $X/\langle c\rangle$, which is homotopy equivalent to a circle, and the resolution of Theorem~\ref{Thm:zg} is built from the cellular chain complex of $X/\langle c\rangle$. In the other direction, suppose that $\G$ has periodic cohomology of period $2$ after $k$ steps. By a result of Talelli \cite[Cor.~3.3]{talelli2005}, the periodicity isomorphisms are then induced by cup product with a class $e\in\Ho^{2}(\G,\Z)$ if and only if $\Ho^{i}(\G,P)=0$ for all $i>k$ and all projective $\Z\G$-modules $P$, and Adem and Smith proved Conjecture~\ref{Conj} under this hypothesis \cite{ademsmith}. Such a class $e$ determines a central extension $1\to\Z\to G_{e}\to\G\to1$, as in Remark~\ref{Rem:gysin}. This raises the question of when $G_{e}$ has finite cohomological dimension. When it does, the construction of this paper applies to $G_{e}$, and it gives eventually $2$-periodic resolutions over $\Z\G$ by Theorem~\ref{Thm:zg}(1).

    The reduction of the matrix factorization of Corollary~\ref{Cor:mfzg} is a $2$-periodic complete resolution in the sense of Talelli, and it computes the complete cohomology of $\G$ in all degrees \cite[Def.~2.1 and Cor.~3.3]{talelli2005}. The cokernels of this complete resolution are Gorenstein projective $\Z\G$-modules \cite[Lem.~5.9]{moriueyama}, and since $\ZG$ has finite global dimension by Lemma~\ref{Lem:gdim}, the work of Chen \cite[Thms.~1 and 2]{chenmf} identifies the stable category of Gorenstein projective $\Z\G$-modules with a stable category of module factorizations of $1-c$ over $\ZG$. This suggests studying Gorenstein projective $\Z\G$-modules by means of factorizations of $1-c$.

    Finally, Sections~\ref{S:gencons}--\ref{S:higher} need only a regular central element of the ring, which in Theorem~\ref{Thm:zg} is $1-c$, and this suggests two generalizations of the setting of Section~\ref{S:applications}. If $\langle c\rangle$ is only normal in $G$, then $1-c$ is a regular normal element of $\ZG$, and $\G$ still acts freely on $X/\langle c\rangle$, but possibly by $-1$ on $\Ho_{1}(X/\langle c\rangle,\Z)\cong\Z$. By \cite[Prop.~5.1]{mislintalelli} the join of $X/\langle c\rangle$ with itself is a free $\G$-complex which is homotopy equivalent to $S^{3}$, and together with the twisted matrix factorizations of \cite{ccmk, chenmf} this suggests resolutions of period $4$. If instead $G$ has a central free abelian subgroup $\langle c_{1},\ldots,c_{m}\rangle$ of rank $m\ge2$, then $1-c_{1},\ldots,1-c_{m}$ is a central regular sequence in $\ZG$ with quotient $\Z[G/\langle c_{1},\ldots,c_{m}\rangle]$, and by analogy with complete intersections \cite[Thms.~4.2.3 and 4.2.5]{ebudpeeva} one expects resolutions whose ranks grow polynomially of degree $m-1$.

    \section*{Acknowledgments}

    This paper constitutes the first part of the author's Ph.D. thesis. The author would like to thank his Ph.D. advisor Alex Martsinkovsky for his constant encouragement, patience, and discussions on the content of this paper. The author would also like to thank Vitor Gulisz for his thorough and very helpful comments, which significantly improved the body of this text. The author is also grateful to an anonymous referee of an earlier version of this paper, whose comments led to a substantial revision of the construction in Section~\ref{S:gencons}.

    \printbibliography
    
\end{document}